\documentclass[10pt,oneside, a4paper,reqno]{amsart}
\usepackage[T1]{fontenc}
\usepackage[utf8]{inputenc}
\usepackage[UKenglish]{babel}

\usepackage[autostyle,italian=guillemets]{csquotes}
\usepackage[style=alphabetic,backend=bibtex,useprefix=true]{biblatex}
\usepackage{pdfpages}
\usepackage{geometry}
\usepackage{enumitem}
\usepackage{microtype}
\usepackage{mathtools}
\usepackage{stmaryrd}
\usepackage{multicol}
\usepackage{comment}
\usepackage{xspace} 
\usepackage{accents}
\usepackage{calc}
\usepackage{fancyhdr}

\usepackage[colorlinks=true,linkcolor=blue, citecolor=red]{hyperref}
\usepackage{colorprofiles}

\usepackage{amscd,amssymb,amsmath,graphicx,verbatim,mathrsfs,xcolor, amsthm}
\usepackage{wasysym}
\usepackage{bbm}
\usepackage{tikz}
\usepackage{tikz-cd}
\usetikzlibrary{calc,positioning}
\usetikzlibrary {decorations.pathmorphing, decorations.pathreplacing, decorations.shapes, fit,patterns,patterns.meta,backgrounds} 
\usepackage[new]{old-arrows}
\tikzset{circle/.style = {rounded corners,line width=1bp,color=#1}}
\usetikzlibrary{babel}

\usepackage{changepage}
\usepackage{graphicx}
\usepackage[capitalize]{cleveref}

\usepackage{quiver}
\usetikzlibrary{arrows.meta, matrix}

\newtheorem{theorem}{Theorem}[section]

\newtheorem{lemma}[theorem]{Lemma}

\theoremstyle{definition}
\newtheorem{definition}[theorem]{Definition}
\newtheorem{example}[theorem]{Example}
\newtheorem{remark}[theorem]{Remark}

\makeatletter
\newtheorem*{rep@theorem}{\rep@title}
\newcommand{\newreptheorem}[2]{%
	\newenvironment{rep#1}[1]{%
		\def\rep@title{#2 \ref{##1}}%
		\begin{rep@theorem}}%
		{\end{rep@theorem}}}
\makeatother

\newreptheorem{theorem}{Theorem}
\newreptheorem{lemma}{Lemma}
\newreptheorem{corollary}{Corollary}

\mathchardef\mhyphen="2D                            
\newcommand{\K}{\mathbf{k}}                         
\newcommand{\Z}{\mathbb{Z}}                         
\newcommand{\op}{^{\mathrm{op}}}                    
\newcommand{\card}[1]{\operatorname{card}(#1)}      
\newcommand{\matspace}[1]{\mathbb{M}_{#1}(\K)}      
\newcommand{\dimvec}[1]{\underline{\dim}(#1)}       
\newcommand{\gldim}[1]{\operatorname{gl.dim}(#1)}   
\newcommand{\sth}{\,\,\vert\,\,}                    
\newcommand{\defeq}{\coloneqq}                      

\DeclareMathOperator{\Hom}{Hom}          
\DeclareMathOperator{\Ext}{Ext}          
\DeclareMathOperator{\End}{End}          
\DeclareMathOperator{\Endol}{endol}      
\DeclareMathOperator*{\colim}{colim}     
\DeclareMathOperator{\Rad}{Rad}          

\newcommand{\Rep}[1]{\operatorname{Rep}_\K(#1)}     
\newcommand{\Repf}[1]{\operatorname{rep}_\K(#1)}    
\newcommand{\Mod}[1]{#1\text{-}\operatorname{Mod}}  
\newcommand{\Modf}[1]{#1\text{-}\operatorname{mod}} 

\newcommand{\qAtilde}{\widetilde{\mathbb{A}}}
\newcommand{\qDtilde}{\widetilde{\mathbb{D}}}
\newcommand{\qEtilde}{\widetilde{\mathbb{E}}}
\newcommand{\qR}{\mathbb{R}}

\newcommand{\palgebra}[1]{\K Q_{#1}}        
\newcommand{\incidencealgebra}[1]{\K{#1}}   
\newcommand{\HasseQuiver}[1]{Q_{#1}}        

\tikzset{
	tableStyles/.style={tips=proper,scale=.5,every node/.style={scale=.8}},
	tableStyles2/.style={tips=proper,scale=.5,every node/.style={scale=.8}},
	nodeDots/.style={fill=black,draw,shape=circle,inner sep=0pt,minimum size=.3em}
}

\usepackage{array}
\usepackage{caption}
\usepackage{tabularx}

\newcommand{\posetAthreetilde}[1]{
	\begin{tikzpicture}[tableStyles,#1]
		\node (min1) at (0,0) [nodeDots] {};
		\node (min2) at (2,0) [nodeDots] {};
		\node (max1) at (1,1) [nodeDots] {};
		\node (max2) at (1,-1) [nodeDots] {};
		
		\draw[-latex] 
		{([xshift=5pt,yshift=5pt]min1.center) edge ([xshift=-5pt,yshift=-5pt]max1.center)}
		{([xshift=5pt,yshift=-5pt]min1.center) edge ([xshift=-5pt,yshift=5pt]max2.center)}
		{([xshift=-5pt,yshift=5pt]min2.center) edge ([xshift=5pt,yshift=-5pt]max1.center)}
		{([xshift=-5pt,yshift=-5pt]min2.center) edge ([xshift=5pt,yshift=5pt]max2.center)}
		;
	\end{tikzpicture}
}

\newcommand{\posetDfourtilde}[1]{
	\begin{tikzpicture}[tableStyles,#1]
		\node (center) at (0,0) [nodeDots] {};
		\node (out1) at (0,1) [nodeDots] {};
		\node (out2) at (1,0) [nodeDots] {};
		\node (out3) at (0,-1) [nodeDots] {};
		\node (out4) at (-1,0) [nodeDots] {};
		\draw[]
		{([yshift=-5pt]out1.center) edge ([yshift=5pt]center.center)}
		{([xshift=-5pt]out2.center) edge ([xshift=5pt]center.center)}
		{([yshift=5pt]out3.center) edge ([yshift=-5pt]center.center)}
		{([xshift=5pt]out4.center) edge ([xshift=-5pt]center.center)}
		;
	\end{tikzpicture}
}

\newcommand{\posetEsixtilde}[1]{
	\begin{tikzpicture}[tableStyles,#1]
		\node (left2) at (0,0) [nodeDots] {};
		\node (left1) at (1,0) [nodeDots] {};
		\node (center) at (2,0) [nodeDots] {};
		\node (above1) at (2,1) [nodeDots] {};
		\node (above2) at (2,2) [nodeDots] {};
		\node (right1) at (3,0) [nodeDots] {};
		\node (right2) at (4,0) [nodeDots] {};
		\draw 
		{([xshift=5pt]left2.center) edge ([xshift=-5pt]left1.center)}
		{([xshift=5pt]left1.center) edge ([xshift=-5pt]center.center)}
		{([xshift=5pt]center.center) edge ([xshift=-5pt]right1.center)}
		{([xshift=5pt]right1.center) edge ([xshift=-5pt]right2.center)}
		{([yshift=-5pt]above2.center) edge ([yshift=5pt]above1.center)}
		{([yshift=-5pt]above1.center) edge ([yshift=5pt]center.center)}
		;
	\end{tikzpicture}
}

\newcommand{\posetEseventilde}[1]{
	\begin{tikzpicture}[tableStyles,#1]
		\node (left3) at (-1,0) [nodeDots] {};
		\node (left2) at (0,0) [nodeDots] {};
		\node (left1) at (1,0) [nodeDots] {};
		\node (center) at (2,0) [nodeDots] {};
		\node (above1) at (2,1) [nodeDots] {};
		\node (right1) at (3,0) [nodeDots] {};
		\node (right2) at (4,0) [nodeDots] {};
		\node (right3) at (5,0) [nodeDots] {};
		\draw 
		{([xshift=5pt]left3.center) edge ([xshift=-5pt]left2.center)}
		{([xshift=5pt]left2.center) edge ([xshift=-5pt]left1.center)}
		{([xshift=5pt]left1.center) edge ([xshift=-5pt]center.center)}
		{([xshift=5pt]center.center) edge ([xshift=-5pt]right1.center)}
		{([xshift=5pt]right1.center) edge ([xshift=-5pt]right2.center)}
		{([xshift=5pt]right2.center) edge ([xshift=-5pt]right3.center)}
		{([yshift=-5pt]above1.center) edge ([yshift=5pt]center.center)}
		;
	\end{tikzpicture}
}

\newcommand{\posetEeighttilde}[1]{
	\begin{tikzpicture}[tableStyles,#1]
		\node (left1) at (1,0) [nodeDots] {};
		\node (center) at (2,0) [nodeDots] {};
		\node (above1) at (2,1) [nodeDots] {};
		\node (above2) at (2,2) [nodeDots] {};
		\node (right1) at (3,0) [nodeDots] {};
		\node (right2) at (4,0) [nodeDots] {};
		\node (right3) at (5,0) [nodeDots] {};
		\node (right4) at (6,0) [nodeDots] {};
		\node (right5) at (7,0) [nodeDots] {};
		\draw 
		
		{([xshift=5pt]left1.center) edge ([xshift=-5pt]center.center)}
		{([xshift=5pt]center.center) edge ([xshift=-5pt]right1.center)}
		{([xshift=5pt]right1.center) edge ([xshift=-5pt]right2.center)}
		{([xshift=5pt]right2.center) edge ([xshift=-5pt]right3.center)}
		{([xshift=5pt]right3.center) edge ([xshift=-5pt]right4.center)}
		{([xshift=5pt]right4.center) edge ([xshift=-5pt]right5.center)}
		{([yshift=-5pt]above2.center) edge ([yshift=5pt]above1.center)}
		{([yshift=-5pt]above1.center) edge ([yshift=5pt]center.center)}
		;
	\end{tikzpicture}
}

\newcommand{\posetRone}[1]{
	\begin{tikzpicture}[tableStyles,#1]
		\node (min) at (0,-1) [nodeDots] {};
		\node (minChildLeft) at (-1,0) [nodeDots] {};
		\node (minChildRight) at (1,0) [nodeDots] {};
		\node[nodeDots] (minGrandChild) at (0,1)  {};
		\node (left1) at (-1,1) [nodeDots] {};
		\node (right1) at (2,0) [nodeDots] {};
		\node (right2) at (3,0) [nodeDots] {};
		\node (right3) at (4,0) [nodeDots] {};
		\node (right4) at (5,0) [nodeDots] {};
		\draw[-latex]
		{([xshift=5pt,yshift=5pt]min.center) edge ([xshift=-5pt,yshift=-5pt]minChildRight.center)}
		{([xshift=-5pt,yshift=5pt]min.center) edge ([xshift=5pt,yshift=-5pt]minChildLeft.center)}
		{([xshift=5pt,yshift=5pt]minChildLeft.center) edge ([xshift=-5pt,yshift=-5pt]minGrandChild)}
		{([xshift=-5pt,yshift=5pt]minChildRight.center) edge ([xshift=5pt,yshift=-5pt]minGrandChild.center)}
		;
		\draw[] 
		{(min.center) edge [dashed] (minGrandChild.center)}
		{([xshift=5pt]left1.center) edge ([xshift=-5pt]minGrandChild.center)}
		{([xshift=5pt]minChildRight.center) edge ([xshift=-5pt]right1.center)}
		{([xshift=5pt]right1.center) edge ([xshift=-5pt]right2.center)}
		{([xshift=5pt]right2.center) edge ([xshift=-5pt]right3.center)}
		{([xshift=5pt]right3.center) edge ([xshift=-5pt]right4.center)}
		;
	\end{tikzpicture}
}

\newcommand{\posetRtwo}[1]{
	\begin{tikzpicture}[tableStyles,#1]
		\node (min) at (0,-1) [nodeDots] {};
		\node (minChildLeft) at (-1,0) [nodeDots] {};
		\node (minChildRight) at (1,0) [nodeDots] {};
		\node (minGrandChild) at (0,1) [nodeDots] {};
		\node (left1) at (-1,1) [nodeDots] {};
		\node (right1) at (2,0) [nodeDots] {};
		\node (left2) at (-2,1) [nodeDots] {};
		\node (left3) at (-3,1) [nodeDots] {};
		\node (left4) at (-4,1) [nodeDots] {};
		\draw[-latex]
		{([xshift=5pt,yshift=5pt]min.center) edge ([xshift=-5pt,yshift=-5pt]minChildRight.center)}
		{([xshift=-5pt,yshift=5pt]min.center) edge ([xshift=5pt,yshift=-5pt]minChildLeft.center)}
		{([xshift=5pt,yshift=5pt]minChildLeft.center) edge ([xshift=-5pt,yshift=-5pt]minGrandChild.center)}
		{([xshift=-5pt,yshift=5pt]minChildRight.center) edge ([xshift=5pt,yshift=-5pt]minGrandChild.center)}
		;
		\draw 
		{(min.center) edge [dashed] (minGrandChild.center)}
		{([xshift=5pt]left4.center) edge ([xshift=-5pt]left3.center)}
		{([xshift=5pt]left3.center) edge ([xshift=-5pt]left2.center)}
		{([xshift=5pt]left2.center) edge ([xshift=-5pt]left1.center)}
		{([xshift=5pt]left1.center) edge ([xshift=-5pt]minGrandChild.center)}
		{([xshift=5pt]minChildRight.center) edge ([xshift=-5pt]right1.center)}
		
		;
	\end{tikzpicture}
}

\newcommand{\posetRthree}[1]{
	\begin{tikzpicture}[tableStyles,#1]
		\node (min) at (0,-1) [nodeDots] {};
		\node (minChildLeft) at (-1,0) [nodeDots] {};
		\node (minChildRight) at (1,0) [nodeDots] {};
		\node (minGrandChild) at (0,1) [nodeDots] {};
		\node (left1) at (-2,0) [nodeDots] {};
		\node (right1) at (2,0) [nodeDots] {};
		\node (left2) at (-3,0) [nodeDots] {};
		\node (right2) at (3,0) [nodeDots] {};
		
		\draw[-latex]
		{([xshift=5pt,yshift=5pt]min.center) edge ([xshift=-5pt,yshift=-5pt]minChildRight.center)}
		{([xshift=-5pt,yshift=5pt]min.center) edge ([xshift=5pt,yshift=-5pt]minChildLeft.center)}
		{([xshift=5pt,yshift=5pt]minChildLeft.center) edge ([xshift=-5pt,yshift=-5pt]minGrandChild.center)}
		{([xshift=-5pt,yshift=5pt]minChildRight.center) edge ([xshift=5pt,yshift=-5pt]minGrandChild.center)}
		;
		\draw 
		{(min.center) edge [dashed] (minGrandChild.center)}
		{([xshift=5pt]left2.center) edge ([xshift=-5pt]left1.center)}
		{([xshift=5pt]left1.center) edge ([xshift=-5pt]minChildLeft.center)}
		{([xshift=5pt]minChildRight.center) edge ([xshift=-5pt]right1.center)}
		{([xshift=5pt]right1.center) edge ([xshift=-5pt]right2.center)}
		;
	\end{tikzpicture}
}

\newcommand{\posetRfour}[1]{
	\begin{tikzpicture}[tableStyles,#1]
		\node (min) at (0,-1) [nodeDots] {};
		\node (minChildLeft) at (-1,0) [nodeDots] {};
		\node (minChildRight) at (1,0) [nodeDots] {};
		\node (minGrandChild) at (0,1) [nodeDots] {};
		\node (left1) at (-2,0) [nodeDots] {};
		\node (right1) at (2,0) [nodeDots] {};
		\node (right2) at (3,0) [nodeDots] {};
		\node (right3) at (4,0) [nodeDots] {};
		\node (right4) at (5,0) [nodeDots] {};
		
		\draw[-latex]
		{([xshift=5pt,yshift=5pt]min.center) edge ([xshift=-5pt,yshift=-5pt]minChildRight.center)}
		{([xshift=-5pt,yshift=5pt]min.center) edge ([xshift=5pt,yshift=-5pt]minChildLeft.center)}
		{([xshift=5pt,yshift=5pt]minChildLeft.center) edge ([xshift=-5pt,yshift=-5pt]minGrandChild.center)}
		{([xshift=-5pt,yshift=5pt]minChildRight.center) edge ([xshift=5pt,yshift=-5pt]minGrandChild.center)}
		;
		\draw 
		{(min.center) edge [dashed] (minGrandChild.center)}
		{([xshift=5pt]left1.center) edge ([xshift=-5pt]minChildLeft.center)}
		{([xshift=5pt]minChildRight.center) edge ([xshift=-5pt]right1.center)}
		{([xshift=5pt]right1.center) edge ([xshift=-5pt]right2.center)}
		{([xshift=5pt]right2.center) edge ([xshift=-5pt]right3.center)}
		{([xshift=5pt]right3.center) edge ([xshift=-5pt]right4.center)}
		;
	\end{tikzpicture}
}

\newcommand{\posetRfive}[1]{
	\begin{tikzpicture}[tableStyles,#1]
		\node (min) at (0,-2) [nodeDots] {};
		\node (minChildLeft) at (-1,-1) [nodeDots] {};
		\node (minGrandChildLeft) at (-1,0) [nodeDots] {};
		\node (minChildRight) at (1,0) [nodeDots] {};
		\node (minGrandGrandChild) at (0,1) [nodeDots] {}; 
		\draw[-latex]
		{([xshift=2pt,yshift=5pt]min.center) edge ([yshift=-5pt]minChildRight.center)}
		{([xshift=-5pt,yshift=5pt]minChildRight.center) edge ([xshift=5pt,yshift=-5pt]minGrandGrandChild.center)}
		{([xshift=-5pt,yshift=5pt]min.center) edge ([xshift=5pt,yshift=-5pt]minChildLeft.center)}
		{([yshift=5pt]minChildLeft.center) edge ([yshift=-5pt]minGrandChildLeft.center)}
		{([xshift=5pt,yshift=5pt]minGrandChildLeft.center) edge ([xshift=-5pt,yshift=-5pt]minGrandGrandChild.center)}
		;
		\node (left1) at (-2,0) [nodeDots] {};
		\node (left2) at (-3,0) [nodeDots] {};
		\node (left3) at (-4,0) [nodeDots] {};
		\node (left4) at (-5,0) [nodeDots] {};
		\draw 
		{(min.center) edge [dashed] (minGrandGrandChild.center)}
		{([xshift=5pt]left4.center) edge ([xshift=-5pt]left3.center)}
		{([xshift=5pt]left3.center) edge ([xshift=-5pt]left2.center)}
		{([xshift=5pt]left2.center) edge ([xshift=-5pt]left1.center)}
		{([xshift=5pt]left1.center) edge ([xshift=-5pt]minGrandChildLeft.center)}
		;
	\end{tikzpicture}
}

\newcommand{\posetRsix}[1]{
	\begin{tikzpicture}[tableStyles,#1]
		\node (min) at (0,-2) [nodeDots] {};
		\node (minChildLeft) at (-1,-1) [nodeDots] {};
		\node (minGrandChildLeft) at (-1,0) [nodeDots] {};
		\node (minChildRight) at (1,0) [nodeDots] {};
		\node (minGrandGrandChild) at (0,1) [nodeDots] {}; 
		\draw[-latex]
		{([xshift=2pt,yshift=5pt]min.center) edge ([yshift=-5pt]minChildRight.center)}
		{([xshift=-5pt,yshift=5pt]minChildRight.center) edge ([xshift=5pt,yshift=-5pt]minGrandGrandChild.center)}
		{([xshift=-5pt,yshift=5pt]min.center) edge ([xshift=5pt,yshift=-5pt]minChildLeft.center)}
		{([yshift=5pt]minChildLeft.center) edge ([yshift=-5pt]minGrandChildLeft.center)}
		{([xshift=5pt,yshift=5pt]minGrandChildLeft.center) edge ([xshift=-5pt,yshift=-5pt]minGrandGrandChild.center)}
		;
		\node (left1) at (-2,0) [nodeDots] {};
		\node (right1) at (2,0) [nodeDots] {};
		\node (right2) at (3,0) [nodeDots] {};
		\node (right3) at (4,0) [nodeDots] {};
		
		\draw 
		{(min.center) edge [dashed] (minGrandGrandChild.center)}
		{([xshift=5pt]left1.center) edge ([xshift=-5pt]minGrandChildLeft.center)}
		{([xshift=5pt]minChildRight.center) edge ([xshift=-5pt]right1.center)}
		{([xshift=5pt]right1.center) edge ([xshift=-5pt]right2.center)}
		{([xshift=5pt]right2.center) edge ([xshift=-5pt]right3.center)}
		;
	\end{tikzpicture}
}

\newcommand{\posetRseven}[1]{
	\begin{tikzpicture}[tableStyles,#1]
		\node (min) at (0,-2) [nodeDots] {};
		\node (minChildLeft) at (-1,-1) [nodeDots] {};
		\node (minGrandChildLeft) at (-1,0) [nodeDots] {};
		\node (minGrandGrandChildLeft) at (-1,1) [nodeDots] {};
		\node (minChildRight) at (1,0) [nodeDots] {};
		\node (minGrandGrandGrandChild) at (0,2) [nodeDots] {}; 
		\draw[-latex]
		{([xshift=2pt,yshift=5pt]min.center) edge ([xshift=-2pt,yshift=-5pt]minChildRight.center)}
		{([xshift=-2pt,yshift=5pt]minChildRight.center) edge ([xshift=2pt,yshift=-5pt]minGrandGrandGrandChild.center)}
		{([xshift=-5pt,yshift=5pt]min.center) edge ([xshift=5pt,yshift=-5pt]minChildLeft.center)}
		{([yshift=5pt]minChildLeft.center) edge ([yshift=-5pt]minGrandChildLeft.center)}
		{([yshift=5pt]minGrandChildLeft.center) edge ([yshift=-5pt]minGrandGrandChildLeft.center)}
		{([xshift=5pt,yshift=5pt]minGrandGrandChildLeft.center) edge ([xshift=-5pt,yshift=-5pt]minGrandGrandGrandChild.center)}
		;
		\node (left1) at (-2,0) [nodeDots] {};
		\node (right1) at (2,0) [nodeDots] {};
		\node (right2) at (3,0) [nodeDots] {};
		
		\draw 
		{(min.center) edge [dashed] (minGrandGrandGrandChild.center)}
		{([xshift=5pt]left1.center) edge ([xshift=-5pt]minGrandChildLeft.center)}
		{([xshift=5pt]minChildRight.center) edge ([xshift=-5pt]right1.center)}
		{([xshift=5pt]right1.center) edge ([xshift=-5pt]right2.center)}
		
		;
	\end{tikzpicture}
}

\newcommand{\Paragraphtitlebf}[1]{\vspace{1em}{\bfseries \noindent #1}}

\title{Incidence algebras and beyond: brick-continuity and generic bricks}

\usepackage{orcidlink}
\usepackage[foot]{amsaddr}

\author{Edoardo Tacchetti}
\address{Edoardo Tacchetti \orcidlink{0009-0000-7446-9942}, Dipartimento di Informatica - Settore di Matematica, Universit\`a degli Studi di Verona, Strada Le Grazie 15, 37134 Verona, Italy}
\email{edoardo.tacchetti@univr.it}

\begin{document}

	\begin{abstract}
		We will show that the twelve families of minimally representation-infinite incidence algebras are brick-continuous and admit a generic brick, which will be explicitly constructed. After that, reversing the reduction process between posets, we will show that any representation-infinite incidence algebra is brick-continuous and admits a generic brick, which can be explicitly computed using what will be called ``extension functors''. Lastly, a method to apply these results to a broader class of algebras is given.
	\end{abstract}
	
	\maketitle
	\vspace{-1em}
	\tableofcontents

	\section*{Introduction}		
	In the last decades, generic modules and generic bricks have caught the interest of researchers thanks to their connections with the representation-type and properties of finite-dimensional algebras. We recall that a module over a finite-dimensional algebra is called \emph{generic} if it is indecomposable with infinite length but finite endolength, that is its length as a module over its endomorphism ring. If the generic module is also a \emph{brick}, that is its endomorphism ring is a division ring, then it is called \emph{generic brick}.

Crawley-Boevey showed that representation-infinite algebras are characterised by the existence of a generic module \cite[Theorem 4.5]{crawley1991tame} and that tame algebras are equivalent to generically tame algebras \cite[Theorem 4.4]{crawley1991tame}.
It was conjectured by Mousavand and Paquette \cite[Conjecture 4.1]{mousavand2025biserial} that for any algebra, being \emph{brick-infinite}, being \emph{brick-continuous} or admitting a generic brick were three equivalent properties (where an algebra is brick-infinite if it admits an infinite family of non-isomorphic bricks, and brick-continuous if it admits an infinite family of non-isomorphic bricks with the same dimension). They proved such result holds for biserial algebras \cite[Corollary 4.5]{mousavand2025biserial}.
Then, Bautista, Pérez and Salmerón proved the equivalence of brick-continuity and existence of a generic brick for tame algebras \cite[Corollary 1.3]{bau:gen}. After that, Schlegel gave an alternative proof of such result \cite[Theorem 5.6]{schlegel2026infinitetautiltingtheory} and proved the full conjecture for algebras whose Krull-Gabriel dimension is defined \cite[Theorem 5.8]{schlegel2026infinitetautiltingtheory}.

In this paper we will discuss the existence of brick-continuous families and generic bricks for representation-infinite incidence algebras of posets. The \emph{incidence algebra} $\incidencealgebra{P}$ of a finite poset $P$ will be defined as the bound path $\mathbf{k}$-algebra $\mathbf{k}Q_P/I_P$, where $Q_P$ is the \emph{Hasse quiver} of $P$ and $I_P$ the ideal of $\mathbf{k}Q_P$ generated by the relations $\{p_1 - p_2 \}_{(p_1,p_2)}$, where we index over all pairs of paths with the same start and endpoints.

After giving the necessary preliminary notions, we will focus on the twelve classes of \emph{minimal representation-infinite incidence algebras} given by the so called \emph{critical} posets \cite{lou:ind, lou:repr}.
Since these twelve families are representation-infinite they admit a generic module by \cite[Theorem 4.5]{crawley1991tame}. Moreover, such algebras are $\tau$-tilting infinite, thanks to \cite[Theorem 2.15]{borve:tau}. Hence, by \cite[Theorem 4.14]{sent:brick}, they admit a brick that is not finitely generated. Thus, we can ask whether such algebras admit a module having both properties, that is a generic brick.

We will show that each class has tame representation type and it is brick-continuous, thus they admit a generic brick (see \cite[Corollary 1.3]{bau:gen} and \cite[Theorem 5.6]{schlegel2026infinitetautiltingtheory}). This process will utilise the concept of Tits quadratic form together with its \emph{radical vectors}. We will give a method to explicitly construct the brick-continuous families and a generic brick for each of the twelve classes and thanks to the properties of APR-tilting functors, we will prove the following result:
\begin{reptheorem}{thm:genbrickcritical}
	\emph{The twelve families of minimally representation-infinite incidence $\K$-algebras given by critical posets in Table \ref{tab:Loupias frames} (considering any orientation of the undirected edges), and also the opposites of these, are brick-continuous and admit a generic brick which can be explicitly constructed.}
\end{reptheorem}
After that, we will further expand our view. Indeed, any representation-infinite poset can be reduced to one of the twelve classes (see Lemma \ref{lemma:infinitereduction}) through the reduction techniques devised by Loupias \cite{lou:ind,lou:repr}. We will show that this reduction process can be ``reversed'' in order to preserve brick-continuity and generic bricks using what we will call \emph{extension functors}. This will lead to our second main result:
\begin{reptheorem}{thm:genbrickinfinite}
	\emph{Any representation-infinite poset $P$ is brick-continuous and admits a generic brick over $\incidencealgebra{P}$ that can be explicitly computed.}
\end{reptheorem}
Lastly, using analogous ideas, we will generalise this result to a wider class of algebras:
\begin{reptheorem}{thm:generalisation}
	\emph{Let $\Lambda, \Lambda'$ be finite-dimensional algebras.
	\begin{enumerate}
		\item If $\Lambda'$ is brick-continuous and $F \colon \Modf{\Lambda'} \to \Modf{\Lambda}$ is a fully faithful exact functor, then $\Lambda$ is brick-continuous.
		\item If $\Lambda'$ has a generic brick and $F \colon \Mod{\Lambda'} \to \Mod{\Lambda}$ is a fully faithful exact functor that preserves endofinite modules, then $\Lambda$ admits a generic brick.
	\end{enumerate}
	}
\end{reptheorem}
A direct application of this result is given when $\Lambda'=\incidencealgebra{P}$ for a representation-infinite poset $P$, for example having a ring epimorphism $f\colon \Lambda \to \incidencealgebra{P}$.

	\Paragraphtitlebf{Acknowledgments.} This paper is based on my Master's thesis \cite{tacchetti2025masterthesis} at University of Padova under the supervision of Jorge Vitória. No generative AI tools were used in the development or writing of this paper.
	The author is grateful to Jorge Vitória for all the help, the advices and the discussions that lead to the present article.	The author also wishes to thank Gustavo Jasso for his valuable comments and suggestions. 
		
	The author is supported by the Program FIS2021 of the Italian Ministry of University and Research (Project LAVIE, FIS00001706). The author is a member of the network INdAM-G.N.S.A.G.A.
	Some of the diagrams were made using \href{https://q.uiver.app/}{q.uiver.app}.
	
	\Paragraphtitlebf{Notation and conventions.} An algebraically closed field $\K$ is fixed throughout. 
	
	We may assume that any finite dimensional $\K$-algebra is basic and connected, so of the form $\palgebra{}/I$ for some bound quiver $(Q,I)$. For a finite-dimensional $\K$-algebra $\Lambda\cong\palgebra{}/I$, $\Mod{\Lambda}$ (resp.~$\Modf{\Lambda}$) will denote the category of (finite dimensional) left $\Lambda$-modules and $\Rep{Q,I}$ (resp.~$\Repf{Q,I}$) the category of (finite-dimensional) representations of the bound quiver $(Q,I)$. We recall that in this case $\Mod{\Lambda}\cong\Rep{Q,I}$ and $\Modf{\Lambda}\cong\Repf{Q,I}$.
	
	Throughout this paper, undirected edges in a quiver indicate that the arrows may go in either direction. Composition of arrows in a quiver will be as composition of functions.
	
	Any poset will be considered to be finite and connected (i.e.~the corresponding bound quiver $(\HasseQuiver{P}, I_P)$, see Definition \ref{def:incidencealgebra}, is connected).
		
	\section{Preliminary notions}	
	\begin{definition}\cite{sim:vol1,sim:vol3}
	A finite dimensional algebra $\Lambda$ is said to be of \textbf{finite representation type} (or \textbf{representation-finite}) if the category $\Modf{\Lambda}$ has a finite number of isomorphism classes of indecomposable modules. Otherwise it is said to be of \textbf{infinite representation type} (or \textbf{representation-infinite}). Moreover $\Lambda$ is said to be:
	\begin{itemize}
		\item of \textbf{tame representation type} (shortly \textbf{tame}) if there is a classification of the isomorphism classes of the indecomposable modules in $\Modf{\Lambda}$ in the sense
		that, for each integer $d \geq 1$, the indecomposable modules in $\Modf{\Lambda}$ of dimension $d$ form at most finitely many one-parameter families.
		\item of \textbf{wild representation type} (shortly \textbf{wild}) if, for any finite dimensional $\K$-algebra $A$, there exists an embedding functor $T\colon \Modf{A}\to \Modf{\Lambda}$.
	\end{itemize}
\end{definition}


\subsection{Tame concealed algebras}
We will work with tame concealed algebras. We recall their definition and some useful properties.

\begin{definition}[{c.f.~\cite[Ch.VIII 3.1, 4.6]{sim:vol1}}, {\cite[Ch.XI.3]{sim:vol2}}]
	Let $Q$ be an acyclic quiver. A finite-dimensional $\K$-algebra $\Lambda$ is said to be \textbf{tilted of type $Q$} if it is isomorphic to $\End_{\K Q}(T)$ for some tilting $\K Q$-module $T$. If $T$ is also postprojective, $\Lambda$ is said to be \textbf{concealed of type $Q$}.
	
	We say that a concealed $\K$-algebra of type $Q$ is \textbf{tame concealed} (equiv.~\textbf{concealed of Euclidean Type}) if $Q$ is a tame and representation-infinite quiver (i.e.~an Euclidean quiver).
\end{definition}

We have the following result, justifying the name ``tame concealed'':
\begin{theorem}\label{thm:tameconcealed}
	A tame concealed algebra is representation-infinite, tame and $\tau$-tilting infinite.
\end{theorem}
\begin{proof}
	It follows directly from the classification of tame concealed algebras that they are representation-infinite (see \cite[Theorem 1]{hap:min}, \cite[Section 4.3, Proposition 7]{ringel2006tame}). They are also tame by \cite[Ch.XIX, Theorem 3.14]{sim:vol3}, and $\tau$-tilting infinite (see \cite[Lemma 2.6]{borve:tau}).
\end{proof}


\subsection{Tits form}
We will need the notion of the Tits form of a bound quiver. We present here some definitions and results that will help us in the following.

\begin{definition}[{c.f.~\cite[Section 1.0]{ringel2006tame}}]
	An \textbf{integral quadratic form} is a polynomial $q=q(x_1,\dots,x_n)$ in $n$ variables with integral coefficients $q_{ij}\in \Z$ of the form
	\[q=q(x_1,\dots,x_n)=\sum_{i=1}^{n} x_i^2 + \sum_{i<j} q_{ij} x_i x_j\]
	We obtain a function $q\colon\Z^n \to \Z$ evaluating the polynomial $q$ at $n$-tuples of integral numbers. 
	
	We can endow $\Z^n$ with a partial order defined component-wise. Given $z = (z_1,\dots,z_n) \in \Z^n$, it is said to be:
	\begin{itemize}
		\item \textbf{positive}, denoted by $z > 0$, if $z\neq 0$ and $z_i\geq0$ for all $i$;
		\item \textbf{sincere} if $z_i\neq 0$ for all $i$.
	\end{itemize}
	
	An integral quadratic form $q$ in $n$ variables is said to be \textbf{positive semi-definite} if $q(z)\geq 0$ for all $z\in\Z^n$.
	
	For a positive semi-definite quadratic form $q$, we can define the \textbf{radical of $q$} as 
	\[\Rad(q)\defeq \{z\in \Z^n\sth q(z) = 0\}\]
	It is a subgroup of $\Z^n$ and its rank is called the \textbf{radical rank of $q$}.
	The elements $z\in\Rad(q)$ are called \textbf{radical vectors}.
	
	In case $0$ is the only radical	vector of $q$ or, equivalently, the radical rank of $q$ is $0$, the form $q$ is said to be \textbf{positive definite}.
\end{definition}

\begin{definition}[{c.f.~\cite[Section 1.1]{ringel2006tame}}]\label{def:quadmatrix}
	Given $q=q(x_1,\dots,x_n)=\sum_{i=1}^{n} x_i^2 + \sum_{i<j} q_{ij} x_i x_j$ an integral quadratic form in $n$ variables, we define the corresponding symmetric bilinear form $\beta\colon \Z^n \times \Z^n \to \Z$ as $\beta(x,y) \defeq x^t B y$ with $B$ the symmetric $n\times n$-matrix $B=\frac{1}{2}(q_{ij})_{1 \leq i,j \leq n}$, with $q_{ij}=q_{ji}$ for $i > j$, and $q_{ii} = 2$ for all $i$.	
\end{definition} 

\begin{lemma}\label{lem:radpossemi}
	Let $q$ be an integral quadratic form, $\beta$ the corresponding symmetric bilinear form with matrix $B$. If $q$ is positive semi-definite then we have that
	\[\Rad(q)=\ker(B)\]
\end{lemma}
\begin{proof}
	($\supseteq$) If $z\in\Z^n$ is in $\ker(B)$ (i.e.~$Bz=0$) then $q(z)=\beta(z,z)=z^t B z = 0$. Thus $z$ is in the radical $\Rad(q)$ of $q$.
	
	($\subseteq$) Let $z$ be in the radical $\Rad(q)$ of $q$, so $q(z)=0$. Since $q$ is positive semi-definite, we have that $\beta$ is a positive semi-definite bilinear form, i.e.~$B$ is a positive semi-definite matrix (all its eigenvalues are $\geq0$). By Cholesky factorisation (see \cite[Corollary 7.2.9]{horn2012matrix}) we can write $B= L^t L$ with $L$ a real lower triangular matrix with non-negative diagonal entries. Therefore
	\[
	0=q(z)=\beta(z,z)=z^t B z= z^t L^t Lz=(Lz)^t Lz = \lVert Lz\rVert ^2
	\]
	with $\lVert\mhyphen\rVert$ being the (real) Euclidean norm ($\lVert x\rVert^2 =x^t x= \sum_{i=1}^{n} x_i^2$). Then $\lVert Lz\rVert ^2=0$, which implies that $Lz=0$. Thus $Bz=L^tLz=0$, so $z\in\ker(B)$.
\end{proof}

\begin{definition}[{c.f.~\cite[Section 1.0]{ringel2006tame}}]
	A quadratic form $q$ in $n$ variables is said to be \textbf{weakly positive} if $q(z) > 0$ for all positive $z\in\Z^n$.
	
	For any $1\leq t\leq n$, we denote by $q^t= q(z_1,\dots, z_{t-1},0,z_{t+1},\dots,z_n)$ the restriction of $q$ to the hyper-plane defined by $z_t=0$.
	
	An integral quadratic form $q$ in $n\geq 3$ variables is said to be \textbf{critical} if $q$ is not weakly positive but all the
	forms $q^t$, $1\leq t\leq n$, are weakly positive.
\end{definition}

We have the following result due to Ovsienko:
\begin{theorem}[{c.f.~\cite[Section 1.0, Theorem 2]{ringel2006tame}}]\label{thm:criticalform}
	A critical quadratic form is positive semi-definite	with radical rank $1$, and with a sincere positive radical vector.
\end{theorem}

Now we define some integral quadratic forms over algebras and quivers.

\begin{definition}[{c.f.~\cite[Section 2.2]{bongartz83quad}}]
	Let $\Lambda$ be a finite dimensional $\K$-algebra with finite global dimension ($\gldim{\Lambda}\leq m$), $\Lambda\cong \palgebra{}/I$ for some finite quiver $Q$ (with $n= \card{Q_0}$) and admissible ideal $I$. We define the \textbf{Euler characteristic} of $\Lambda$ as the bilinear form 
	\[
	\langle \dimvec{M},\dimvec{M'}\rangle \defeq \sum_{i\geq0} (-1)^i \dim_\K \Ext_{\Lambda}^i(M,M')  
	\]
	where $M$, $M'$ are finitely generated $\Lambda$-modules. It is well defined since $\gldim{\Lambda}\leq m$, so $\Ext_{\Lambda}^i(M,M')=0$ for $i\geq m+1$. We define the \textbf{Euler quadratic form} of $\Lambda$ as
	\[\chi_\Lambda(\dimvec{M})\defeq \langle \dimvec{M},\dimvec{M}\rangle\]
\end{definition}

\begin{definition}[{c.f.~\cite[Definition 2.1]{bongartz83quad}}]
	Let $(Q,I)$ be an arbitrary bound quiver with $n=\card{Q_0}$ and $Q$ acyclic. The \textbf{Tits quadratic form} of $\Lambda=\palgebra{}/I$ is the integral quadratic form $q_\Lambda \colon \Z^{Q_0} \to \Z$ defined by the formula:
	\[
	\begin{aligned}
		q_\Lambda(x)\defeq& \sum_{i\in Q_0} x_i^2 - \sum_{\alpha\in Q_1} x_{s(\alpha)} x_{t(\alpha)} + \sum_{i,j \in Q_0} r_{ij} x_i x_j \\
		=& \sum_{i\in Q_0} x_i^2 - \sum_{i,j\in Q_0} d_{ij} x_i x_j + \sum_{i,j \in Q_0} r_{ij} x_i x_j
	\end{aligned}
	\]
	where $d_{ij}$ is the number of arrows from $i$ to $j$ and $r_{ij} = \card{R\cap \varepsilon_j \palgebra{} \varepsilon_i}$, with $R$ a minimal set of relations which generates $I$ and $\varepsilon_j \palgebra{} \varepsilon_i$ the vector space spanned by all the paths starting at $i$ and ending in $j$ ($\varepsilon_\ell$ denotes the lazy path at vertex $\ell$) .
\end{definition}

These two integral quadratic forms may be different in general, but they coincide with the following hypothesis:

\begin{theorem}[{c.f.~\cite[Proposition 2.2]{bongartz83quad}}]\label{thm:eulertits}
	Let $\Lambda =\palgebra{}/I$ be a basic finite-dimensional $\K$-algebra  with $\gldim{\Lambda}\leq 2$ and $Q$ acyclic.  Then $\chi_\Lambda$ and $q_\Lambda$ coincide.
\end{theorem}


\subsection{Incidence algebras}

We will now present some notions related to the representation theory of a finite poset $P$.

\begin{definition}[{c.f.~\cite[Definition 1.1]{borve:tau}}]\label{def:incidencealgebra}
	Let $P$ be a finite poset. Its \textbf{Hasse quiver} $\HasseQuiver{P}$ is defined as the quiver with the elements in $P$ as vertices and arrows $x\rightarrow y$ if $x\leq y$ in $P$ and no element lies strictly between $x$ and $y$.
	
	We define the \textbf{incidence algebra} $\incidencealgebra{P}$ of a finite poset $P$ as the bound path algebra $\palgebra{P} / I_P$, where the ideal $I_P$ of the path $\K$-algebra $\palgebra{P}$ is generated by the relations $\{p_1-p_2\}_{p_1,p_2}$, indexing over all pairs of paths with the same start and endpoints.
\end{definition}

 In the following, (finite-dimensional) left $\incidencealgebra{P}$-modules will also be considered as (finite-dimensional) representations of the bound quiver $(\HasseQuiver{P}, I_P)$, namely $\Mod{\incidencealgebra{P}}\cong \Rep{\HasseQuiver{P}, I_P}$ (resp.~$\Modf{\incidencealgebra{P}}\cong \Repf{\HasseQuiver{P}, I_P}$).

\begin{remark}
	We would like to bring to the reader's attention that we will be working with the classical definition of representations of the bound quiver $(\HasseQuiver{P}, I_P)$.
	
	Other authors used the term \emph{representation} of a poset $P$ (or also \emph{$P$-space}) to denote a $\K$-vector space $V$ with an order-preserving map from $P$ into the poset of subspaces of $V$ (see for example \cite{nazarovaroiter:representations, gabriel:representations}). The maps in such representations of the bound quiver $(\HasseQuiver{P},I_P)$ are all injective. This property is not required in the following.
\end{remark}

\begin{definition}
	A finite poset $P$ is said to be \textbf{representation-infinite} if the incidence algebra $\incidencealgebra{P}$ is representation-infinite. The definition of a representation-finite, tame or wild poset is analogous.
\end{definition}

\subsubsection{Critical posets and reduction techniques}
 
Loupias \cite{lou:ind,lou:repr} devised two reduction techniques of posets that proved to be very useful to study the representation theory of incidence algebras.
\begin{definition}[{\cite{lou:ind}, c.f.~\cite[Section 2.3]{borve:tau}}]\label{def:reduction}
	Let $P$ be a finite poset.
	\begin{itemize}
		\itemsep0em
		
		\item A subposet $P'$ of $P$ (with the induced order and the canonical order-preserving map $s\colon P'\to P$) is also called a \textbf{subtraction} of $P$.
		
		\item A poset $P'$ is a \textbf{contraction} (or a \textbf{contracted poset}) of $P$ if there exists a surjective order-preserving morphism $c\colon P\rightarrow P'$ (also called \textbf{contraction}) with connected fibers (i.e.~$c^{-1}(x)$ is connected for every $x\in P'$).
		
		\noindent If there exists exactly one element $x\in P'$ such that $\vert c^{-1}(x)\vert=2$ and $|c^{-1}(y)|=1$ for all $y\in P'\setminus \{x\}$, the contraction $c\colon P\to P'$ is called \textbf{elementary contraction}.
		
		\item A poset $P$ \textbf{can be reduced to $P'$} if $P'$ can be obtained from $P$ using a finite number of subtractions and/or contractions.
	\end{itemize}
\end{definition}
Informally, an elementary contraction identifies two adjacent elements in $P$. Therefore, any arrow in the Hasse quiver of $P$ corresponds to an elementary contraction. It is possible to check that every contraction can be written as a composite of elementary contractions (c.f.~\cite[Lemma 2.29]{tacchetti2025masterthesis}).

If a poset $P$ can be reduced to $P'$, we may assume that each step of the reduction process either removes a single vertex or contracts an arrow in the Hasse quiver.

\begin{definition}[{\cite[Section 1]{lou:ind}}]\label{def:criticalposet}
	A poset $P$ is said to be \textbf{critical} (or $\incidencealgebra{P}$ is said to be \textbf{minimally representation-infinite}) if it is representation-infinite and all of its	proper subposets and non-trivial contractions have representation-finite incidence $\K$-algebra.
\end{definition}

\begin{lemma}\label{lemma:infinitereduction}
	Any representation-infinite poset can be reduced to a critical poset.
\end{lemma}
\begin{proof}
	Let $P$ be a representation-infinite poset. If all of its proper subposets and non-trivial contractions have representation-finite incidence $\K$-algebras, it is critical by Definition \ref{def:criticalposet}.
	Otherwise, $P$ will have at least one proper subposet or non-trivial contraction $P'$ that is representation-infinite. The cardinality of $P'$ is strictly less than the cardinality of $P$. Therefore by induction on the cardinality of the poset, $P'$ can be reduced to a critical poset, thus the same holds for $P$, since we can consider the reduction techniques used from $P$ to $P'$ and from $P'$ to a critical poset in order to reduce $P$ to a critical poset.
\end{proof}

\captionsetup[table]{skip=10pt}
\begin{table}[ht!]
	\centering
	\begin{tabular}{cc|cc|cc}
		$\qAtilde_3$& \posetAthreetilde{} & $\qDtilde_4$  & \posetDfourtilde{} & $\qEtilde_6$& \posetEsixtilde{} \\[3pt] \hline
		&&&&& \\[-6pt]
		$\qEtilde_7$ & \posetEseventilde{} & $\qEtilde_8$& \posetEeighttilde{} & $\qR_1$ & \posetRone{} \\[3pt]  \hline
		&&&&& \\[-6pt]
		$\qR_2$& \posetRtwo{} & $\qR_3$ & \posetRthree{} & $\qR_4$& \posetRfour{} \\[3pt] \hline
		&&&&& \\[-10pt]
		$\qR_5$ & \posetRfive{} &$\qR_6$ & \posetRsix{} & $\qR_7$ &  \posetRseven{} \\[3pt] 
	\end{tabular}
	\caption{Critical posets are given by these bound Hasse quivers (and their opposites). Dashed lines indicate relations. Undirected edges in the quiver indicate that the arrows may go in either direction.}
	\label{tab:Loupias frames}
\end{table}

Loupias showed that the critical posets are exactly the posets in Table \ref{tab:Loupias frames} or the opposites of those (see \cite[Theorem 1.4]{lou:ind}) and that they characterise representation-infinite posets through the reduction process (see \cite[Theorem 1.5]{lou:ind}). Therefore we can summarise these results with the following:

\begin{theorem}[{see \cite[Theorem 1.4, 1.5]{lou:ind}}] \label{thm:Lou75concealed}
	A finite poset $P$ is representation-infinite if and only if it can be reduced to a critical poset $C$, i.e. to a poset whose bound Hasse quiver (or its opposite) is in Table \ref{tab:Loupias frames}. 
\end{theorem}

\begin{lemma}\label{lemma:postameconc}
	Let $C$ be a critical poset. Then $\incidencealgebra{C}$ is tame concealed.
	In particular it is representation-infinite, tame and $\tau$-tilting infinite.
\end{lemma}
\begin{proof}
	Critical posets are given by bound Hasse quivers quivers in Table \ref{tab:Loupias frames} (and their opposites). Such quivers appear on Happel-Vossieck’s list of tame concealed algebras \cite{hap:min}. The second part of the statement follows immediately from Theorem \ref{thm:tameconcealed}.
\end{proof}

Thanks to this result, the incidence algebra $\incidencealgebra{C}$ of a critical poset $C$ is in particular tilted of Euclidean type (by definition), so the following result holds:

\begin{theorem}\label{thm:titscritical}
	Let $C$ be a critical poset, $\incidencealgebra{C}$ its incidence algebra.
	\begin{enumerate}
		\item[$(1)$] $\gldim{\incidencealgebra{C}}\leq 2$ and any indecomposable $X\in \Modf{\incidencealgebra{C}}$ has projective dimension or injective dimension less or equal than one.
		
		\item[$(2)$] $\HasseQuiver{C}$ is acyclic.
		
		\item[$(3)$] The Tits form $q_{\incidencealgebra{C}}$ (also denoted by $q_C$ in the following) is critical. In particular it is positive semi-definite with radical rank $1$ and with a sincere positive radical vector.
	\end{enumerate} 
\end{theorem}
\begin{proof}
	$(1)$ Follows from \cite[Ch.VIII, Lemma 3.2]{sim:vol1} since $\incidencealgebra{C}$ is tilted of Euclidean type.
	
	$(2)$ The quiver associated to a tilted algebra is acyclic (see \cite[Ch.VIII, Corollary 3.4]{sim:vol1}).
	
	$(3)$ By the proof of \cite[Proposition 4.3.7]{ringel2006tame}, tame concealed algebras have a critical Euler quadratic form, therefore $\chi_{\incidencealgebra{C}}$ (also denoted by $\chi_C$) is critical. But now, since $(1)$ and $(2)$ hold, we can apply Theorem \ref{thm:eulertits} and we have that $\chi_C=q_C$. Therefore the Tits form $q_C$ is critical and by Theorem \ref{thm:criticalform} we conclude that it is positive semi-definite with radical rank $1$ and with a sincere positive radical vector.
\end{proof}

	\section{Generic bricks over incidence algebras}
	\begin{definition}[{c.f.~\cite[Introduction]{bau:gen}}]
	Let $\Lambda$ be a finite dimensional $\K$-algebra.
	\begin{itemize}
		\item $\Lambda$ is said to be \textbf{brick-continuous} if it admits an infinite family of non-isomorphic bricks with the same dimension.
		\item The \textbf{endolength} of a module $M\in \Mod{\Lambda}$, denoted by $\Endol(M)$, is the length of $M$ when considered as an $\End_\Lambda(M)\op$-module. A module with finite endolength is called \textbf{endofinite}.
		\item A module $G\in \Mod{\Lambda}$ is called \textbf{generic module}	if it is indecomposable, endofinite and with infinite length as a $\Lambda$-module.
		\item A generic module G is called \textbf{generic brick} if it is also a brick, i.e.~$\End_\Lambda(G)$ is a division algebra over $\K$.
	\end{itemize}
\end{definition}

In the brick-continuous definition we can substitute \emph{same dimension} with \emph{same dimension vector}. Indeed:
\begin{lemma}\label{lemma:brickcontdef}
	Let $\Lambda$ be a finite dimensional $\K$-algebra. Then, $\Lambda$ is brick-continuous if and only if it admits an infinite family of non-isomorphic bricks with the same dimension vector.
\end{lemma}
\begin{proof}
	$(\Leftarrow)$ It is clear, since same dimension vector $\underline{d}=(d_1,\dots,d_n)$ implies same dimension $d=\sum_{i=1}^{n} d_i$ (c.f.~\cite[Ch.III, Corollary 3.6]{sim:vol1}).
	
	$(\Rightarrow)$ Let $\{B_i\}_{i\in I}$ be the brick continuous family of dimension $d$ and for any $i\in I$ let $\underline{d}^i=(d^i_1,\dots,d^i_n)\defeq \dimvec{B_i}$. As before we have that $d=\sum_{j=1}^{n} d^i_j$  for all $i\in I$, so the dimension vectors $\underline{d}^i$ can be thought as partitions of $d$ with at most $n$ summands. Since we have a finite number of such partitions but $I$ is infinite, by the \emph{Infinite Pigeonhole Principle} we must have an infinite subset $H\subseteq I$ such that $\{B_h\}_{h\in H}$ have the same dimension vector.
\end{proof}

For tame algebras we have the following result:
\begin{lemma}
	If $\Lambda$ has tame representation type and $G$ is a generic $\Lambda$-module, then $G$ is a generic brick if and only if $\End_\Lambda(G)\cong \K(x)$ (where $\K(x)$ denotes the field of rational functions, i.e.~the field of fractions of the polynomial ring $\K[x]$).
\end{lemma}
\begin{proof}
	If $G$ is a generic brick, $\End_\Lambda(G)$ is a division algebra over $\K$ and in this case $\Rad(\End_\Lambda(G))=0$. Therefore, by \cite[Theorem 4.6]{crawley1991tame} we get that $\End_\Lambda(G)\cong \K(x)$.
	
	Conversely, if $\End_\Lambda(G)\cong \K(x)$, $G$ is a generic module  that is also a brick, so it is a generic brick by definition.
\end{proof}

\begin{remark}\label{rem:genmodulek(x)}
	Notice that $G$ generic module with $\End_\Lambda(G)\cong \K(x)$ implies $G$ generic brick for any finite-dimensional $\K$-algebra $\Lambda$ of infinite representation type (not necessarily tame).
\end{remark}


\subsection{Generic bricks of critical posets}\label{section:brickcritical}
We will now present the general construction of a brick-continuous family and a generic brick for all the twelve families of critical posets in Table \ref{tab:Loupias frames}, i.e.~for all the minimally representation-infinite incidence $\K$-algebras. The idea we will apply is the following for all the different cases: let $C$ be a critical poset, $n=\card{C}$
\begin{itemize}
	\item by Theorem \ref{thm:titscritical} we know that $q_C$ is critical, in particular it is positive semi-definite with radical rank $1$ and with a sincere positive radical vector. We find $\Rad(q_C)$ using Lemma \ref{lem:radpossemi} through Gauss-Jordan elimination method on the matrix $B_C$ (matrix associated to the Tits form $q_C$, i.e.~$q_C(z)= z^t B_C z$, see Definition \ref{def:quadmatrix}) and we denote by $y_C$ the minimal sincere positive radical vector (in $\Z^n$);
	\item we define a family of representations $\{B(\lambda)_C\}_{\lambda\in \K}$ with $\dimvec{B(\lambda)_C} = y_C$ and we show that for $\lambda\in \K$ we obtain an infinite family of non-isomorphic bricks with the same dimension, proving that $\incidencealgebra{C}$ is brick-continuous;
	\item we explicitly construct the generic brick $G_C$ (that exists thank to \cite[Corollary 1.3]{bau:gen} and \cite[Theorem 5.6]{schlegel2026infinitetautiltingtheory}) starting from the brick-continuous family $\{B(\lambda)_C\}_{\lambda\in \K}$.
\end{itemize}

We will only consider some particular orientations of the critical posets thanks to the following result:

\begin{lemma}\label{lem:orientatonreflection}
	Let $P$ be a poset with Hasse quiver
	\[ 
	\HasseQuiver{P}\defeq \quad \begin{tikzcd}[ampersand replacement=\&,cramped]
		{\HasseQuiver{P'}} \& {a_1} \& {a_2} \& \cdots \& {a_n}
		\arrow[no head, from=1-1, to=1-2]
		\arrow[no head, from=1-2, to=1-3]
		\arrow[no head, from=1-3, to=1-4]
		\arrow[no head, from=1-4, to=1-5]
	\end{tikzcd}\]
	where $P'$ is a subposet with a fixed orientation containing all the relations in $P$ (i.e.~$I_P=I_{P'}$), while the undirected edges denote that the arrows may go in either direction.
	
	Then, there is a bijection of (non-simple) finite-dimensional bricks over any two different orientations of the undirected edges of the bound Hasse quiver $(\HasseQuiver{P}, I_P)$.
	Moreover, there is a bijection of finite-dimensional bricks over any given orientation and the opposite of it.
	
	In particular, if there is a brick-continuous family over any given orientation, there is a brick-continuous family for all the other possible orientations and their opposites.
\end{lemma}

\begin{proof}
	Let $O_1$ and $O_2$ be two different orientations of the undirected edges in $\HasseQuiver{P}$ and let $(\HasseQuiver{P, O_i}, I_P)$ be the bound Hasse quiver with the fixed orientation of the arrows.
	
	The first part of this Lemma follows immediately from the properties of \emph{APR-tilting functors} (introduced in \cite{auslander1979coxeter} as a generalisation of reflection functors \cite{bernstein1973coxeter}).
	Since we only consider changing the orientations of the edge connecting $\HasseQuiver{P'}$ to $a_1$ and the edges between the $a_i$'s, the definition of APR-tilting functors coincide with that of reflection functors in this case.
	Hence, by \cite[Theorem 1.2]{bernstein1973coxeter} we can link the two different orientations of the bound quiver using APR-tilting functor and we obtain a bijection between non-simple finite-dimensional indecomposable modules over $(\HasseQuiver{P, O_1}, I_P)$ and over $(\HasseQuiver{P, O_2},I_P)$.	
	Moreover, by \cite[Theorem 1.11, 1.12]{auslander1979coxeter} these functors are fully faithful, so we can restrict the bijection to non-simple finite-dimensional bricks. Sending the simple representations $S(j)$ over $(\HasseQuiver{P, O_1}, I_P)$ to the simple representations $S(j)$ over $(\HasseQuiver{P, O_2}, I_P)$, we get a bijection between all finite-dimensional bricks over the two different orientations.
	
	The second part of the Lemma follows directly from the properties of duality $D$ of finite-dimensional representations of a bound quiver. Let $O$ be an orientation of the undirected edges in $\HasseQuiver{P}$ and let $(\HasseQuiver{P, O}, I_P)$ be the bound Hasse quiver with the fixed orientation of the edges. The duality $D$, by fully faithfulness, gives a bijection of finite-dimensional bricks over $(\HasseQuiver{P, O}, I_P)$ and over $(\HasseQuiver{P\op, O\op}, I_{P\op})$.

	Lastly, if we have a brick-continuous family of dimension vector $\underline{d}$ for a given orientation, applying APR-tilting functors results in a brick-continuous family with dimension vector $\underline{d}'$, where $\underline{d}'$ is computed following the formulas in \cite[Theorem 1.1]{bernstein1973coxeter}. Instead, if we consider the opposite orientation, we obtain a brick-continuous family with the same dimension vector (by duality of representations).
	
	Therefore changing orientations of the undirected edges and/or considering the opposite bound quivers does not affect the brick-continuity of the bound Hasse quiver.
\end{proof}

\begin{remark}\label{rem:allorientations}
	Thanks to this result we can study only one particular orientation of the critical posets. Indeed if we find a brick-continuous family for one particular orientation of a critical poset, applying Lemma \ref{lem:orientatonreflection} we get brick-continuous families for any orientation of the undirected edges and all their opposite posets.
	
	Our construction of the generic bricks for the critical posets will only depend on the brick-continuous family, so we can find a generic brick for all the possible orientations of the undirected edges of the critical posets and all the opposites of these.
\end{remark}
 
We will now state our first main result:

\begin{theorem}\label{thm:genbrickcritical}
	The twelve families of minimally representation-infinite incidence $\K$-algebras given by critical posets in Table \ref{tab:Loupias frames} (considering any orientation of the undirected edges), and also the opposites of these, are brick-continuous and admit a generic brick which can be explicitly constructed.
\end{theorem}

\begin{proof}
	We only provide the proof for one of the families. The computations for all the cases can be found in \cite[Appendix]{tacchetti2025masterthesis}. Anyway, the proof is similar and can be checked following the same procedure. In Appendix \ref{appendix}, we list the generic bricks (and the brick-continuous families) of critical poset (with a fixed orientation) constructed with our method.
	
	Consider the critical poset of type $P=\qR_3$ (see Table \ref{tab:Loupias frames}) with the following fixed orientation and labelling:
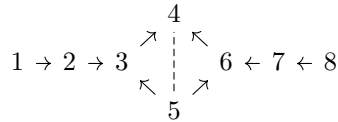
\[\begin{tikzcd}[ampersand replacement=\&,sep=small]
	\&\&\& 4 \\
	1 \& 2 \& 3 \&\& 6 \& 7 \& 8 \\
	\&\&\& 5
	\arrow[from=2-1, to=2-2]
	\arrow[from=2-2, to=2-3]
	\arrow[from=2-3, to=1-4]
	\arrow[from=2-5, to=1-4]
	\arrow[from=2-6, to=2-5]
	\arrow[from=2-7, to=2-6]
	\arrow[dashed, no head, from=3-4, to=1-4]
	\arrow[from=3-4, to=2-3]
	\arrow[from=3-4, to=2-5]
\end{tikzcd}\]
Its Tits form $q_{\qR_3}$ is
\[
q_{\qR_3}(x) = \sum_{i=1}^{8} x_i^2 -x_1 x_2 - x_2 x_3 - x_3 x_4 - x_5 x_3 -x_5 x_6 - x_6 x_4 - x_7 x_6 -x_8 x_7 + x_5 x_4
\]
because $I_{\qR_3} = \langle \alpha_{34}\alpha_{53} - \alpha_{64}\alpha_{56}\rangle$ so $r_{54}=1$. The associated matrix $B_{\qR_3}$ is given by
\[
B_{\qR_3} = \frac{1}{2}
\small
\begin{pmatrix}
	2 & -1 & 0 & 0 & 0 & 0 & 0 & 0  \\
	-1 & 2 & -1 & 0 & 0 & 0 & 0 & 0  \\
	0 & -1 & 2 & -1 & -1 & 0 & 0 & 0  \\
	0 & 0 & -1 & 2 & 1 & -1 & 0 & 0  \\
	0 & 0 & -1 & 1 & 2 & -1 & 0 & 0  \\
	0 & 0 & 0 & -1 & -1 & 2 & -1 & 0  \\
	0 & 0 & 0 & 0 & 0 & -1 & 2 & -1  \\
	0 & 0 & 0 & 0 & 0 & 0 & -1 & 2  
\end{pmatrix}
\quad \xrightarrow[\text{elimination}]{\text{Gauss-Jordan}} \quad
\begin{pmatrix}
	1 & 0 & 0 & 0 & 0 & 0 & 0 & -1 \\
	0 & 1 & 0 & 0 & 0 & 0 & 0 & -2 \\
	0 & 0 & 1 & 0 & 0 & 0 & 0 & -3 \\
	0 & 0 & 0 & 1 & 0 & 0 & 0 & -2 \\
	0 & 0 & 0 & 0 & 1 & 0 & 0 & -2 \\
	0 & 0 & 0 & 0 & 0 & 1 & 0 & -3 \\
	0 & 0 & 0 & 0 & 0 & 0 & 1 & -2 \\
	0 & 0 & 0 & 0 & 0 & 0 & 0 & 0 
\end{pmatrix}
\]
We have a kernel of dimension $1$ generated by $y_{\qR_3} = (1,2,3,2,2,3,2,1)^t$. Thus, by Lemma \ref{lem:radpossemi} $\Rad(q_{\qR_3})=\ker(B_{\qR_3}) = \langle y_{\qR_3}\rangle$.
We define the representations $B(\lambda)_{\qR_3}$ (with $\lambda\in \K$) as follows:
\[B(\lambda)_{\qR_3} \defeq 
\begin{tikzcd}[ampersand replacement=\&,sep=small]
	\&\&\& {\K^2} \&\&\& \\
	\K \& {\K^2} \& {\K^3} \&\& {\K^3} \& {\K^2} \& \K \\
	\&\&\& {\K^2}
	\arrow["{\phi_{1,1}}", from=2-1, to=2-2]
	\arrow["{\phi_{2,1}}", from=2-2, to=2-3]
	\arrow["{\beta_\lambda}", from=2-3, to=1-4]
	\arrow["\alpha"', from=2-5, to=1-4]
	\arrow["{\psi_{2,1}}"', from=2-6, to=2-5]
	\arrow["{\psi_{1,1}}"', from=2-7, to=2-6]
	\arrow[dashed, no head, from=3-4, to=1-4]
	\arrow["{\psi_{2,1}}", from=3-4, to=2-3]
	\arrow["{\phi_{2,1}}"', from=3-4, to=2-5]
\end{tikzcd}\]
where $\alpha \defeq \left( \begin{smallmatrix} 1 & 1 & 0 \\ 0 & 1 & 1 \end{smallmatrix} \right)$, $\beta_\lambda \defeq \left( \begin{smallmatrix} \lambda & 1 & 1 \\ 1 & 0 & 1	\end{smallmatrix} \right)$, $\phi_{n,m}\defeq \left( \begin{smallmatrix} \boldsymbol{1}_n \\ \boldsymbol{0}_{m,n}\end{smallmatrix} \right)$ and $\psi_{n,m}\defeq \left( \begin{smallmatrix} \boldsymbol{0}_{m,n} \\ \boldsymbol{1}_n \end{smallmatrix} \right)$ (with $\boldsymbol{1}_n \in \matspace{n}$ identity matrix, $\boldsymbol{0}_{m,n}\in \matspace{m\times n}$ zero matrix).
The relations are preserved since $\beta_\lambda\psi_{2,1} = 
\left( \begin{smallmatrix}
	1 & 1\\
	0 & 1 
\end{smallmatrix}\right)
=\alpha\phi_{2,1}$.

We have $\dimvec{B(\lambda)_{\qR_3}}=y_{\qR_3}$ and we prove the following:

\begin{lemma}\label{lemma:brickcontR3}
	The representations $B(\lambda)_{\qR_3}$ with $\lambda\in \K$ (with $\K$ algebraically closed field) form an infinite family of non-isomorphic bricks with the same dimension. Therefore, the incidence $\K$-algebra $\incidencealgebra{\qR_3}$ is brick-continuous.
\end{lemma}
\begin{proof}
	Let $f\colon B(\lambda)_{\qR_3} \to B(\lambda')_{\qR_3}$ be a morphism of representations.
	\[\begin{tikzcd}[ampersand replacement=\&,cramped]
		\&\&\&\& {\K^2} \&\&\& \\
		\K \& {\K^2} \& {\K^3} \&\&\& {\K^3} \& {\K^2} \& \K \\
		\&\&\& {\K^2} \\
		\&\&\&\& {\K^2} \\
		\K \& {\K^2} \& {\K^3} \&\&\& {\K^3} \& {\K^2} \& \K \\
		\&\&\& {\K^2}
		\arrow["{f_4}"'{pos=0.4}, color={rgb,255:red,255;green,51;blue,51}, from=1-5, to=4-5]
		\arrow["{\phi_{1,1}}", from=2-1, to=2-2]
		\arrow["{f_1}"', color={rgb,255:red,255;green,51;blue,51}, from=2-1, to=5-1]
		\arrow["{\phi_{2,1}}", from=2-2, to=2-3]
		\arrow["{f_2}"', color={rgb,255:red,255;green,51;blue,51}, from=2-2, to=5-2]
		\arrow["{\beta_\lambda}", from=2-3, to=1-5]
		\arrow["{f_3}"', color={rgb,255:red,255;green,51;blue,51}, from=2-3, to=5-3]
		\arrow["\alpha"', from=2-6, to=1-5]
		\arrow["{f_6}", color={rgb,255:red,255;green,51;blue,51}, from=2-6, to=5-6]
		\arrow["{\psi_{2,1}}"', from=2-7, to=2-6]
		\arrow["{f_7}", color={rgb,255:red,255;green,51;blue,51}, from=2-7, to=5-7]
		\arrow["{\psi_{1,1}}"', from=2-8, to=2-7]
		\arrow["{f_8}", color={rgb,255:red,255;green,51;blue,51}, from=2-8, to=5-8]
		\arrow[dashed, no head, from=3-4, to=1-5]
		\arrow["{\psi_{2,1}}"{pos=0.3}, from=3-4, to=2-3]
		\arrow["{\phi_{2,1}}"'{pos=0.7}, from=3-4, to=2-6]
		\arrow["{f_5}"{pos=0.3}, color={rgb,255:red,255;green,51;blue,51}, from=3-4, to=6-4]
		\arrow["{\phi_{1,1}}", from=5-1, to=5-2]
		\arrow["{\phi_{2,1}}", from=5-2, to=5-3]
		\arrow["{\beta_{\lambda'}}"{pos=0.4}, from=5-3, to=4-5]
		\arrow["\alpha"', from=5-6, to=4-5]
		\arrow["{\psi_{2,1}}"', from=5-7, to=5-6]
		\arrow["{\psi_{1,1}}"', from=5-8, to=5-7]
		\arrow[dashed, no head, from=6-4, to=4-5]
		\arrow["{\psi_{2,1}}", from=6-4, to=5-3]
		\arrow["{\phi_{2,1}}"', from=6-4, to=5-6]
	\end{tikzcd}\]
	We have that $f_1, f_8$ are just scalar multiplications of elements of $\K$, while $f_2,f_4,f_5,f_7 \in \matspace{2}$ and $f_3,f_6 \in \matspace{3}$. We want to understand the property that $f$ must have, studying the commutativity of the squares in the previous diagram.\smallskip
	
	\noindent $(1)$ The square involving $f_1$ and $f_2$ gives us
	$
	f_2(e_1) =
	\begin{pmatrix}
		f_1 \\
		0
	\end{pmatrix}
	$
	\smallskip
	
	\noindent $(2)$ The square involving $f_2$ and $f_3$ gives us that $\phi_{2,1}f_2(e_i) =f_3(\phi_{2,1}(e_i))$ that means
	\[f_3(e_1) = 
	\begin{pmatrix}
		f_2(e_1) \\
		0
	\end{pmatrix} = \begin{pmatrix}
		f_1 \\
		0 \\
		0
	\end{pmatrix} \qquad 
	f_3(e_2) = 
	\begin{pmatrix}
		f_2(e_2) \\
		0
	\end{pmatrix}\]
	
	\noindent $(3)$ The square involving $f_7$ and $f_8$ gives us
	$
	f_7(e_2) = 
	\begin{pmatrix}
		0 \\
		f_8
	\end{pmatrix}
	$
	\smallskip	
	
	\noindent $(4)$ The square involving $f_6$ and $f_7$ gives us that $\psi_{2,1}f_7(e_i) =f_6(\psi_{2,1}(e_i))$ that means
	\[f_6(e_2) = 
	\begin{pmatrix}
		0 \\
		f_7(e_1)
	\end{pmatrix} \qquad 
	f_6(e_3) = 
	\begin{pmatrix}
		0 \\
		f_7(e_2)
	\end{pmatrix}= \begin{pmatrix}
		0 \\
		0 \\
		f_8
	\end{pmatrix}
	\]
	
	\noindent $(5)$ The square involving $f_3$ and $f_5$ gives us that $\psi_{2,1}f_5(e_i) =f_3(\psi_{2,1}(e_i))$ that means
	\[f_3(e_2) = 
	\begin{pmatrix}
		0 \\
		f_5(e_1)
	\end{pmatrix} = \begin{pmatrix}
	f_2(e_2) \\
	0
	\end{pmatrix} =\begin{pmatrix}
	0 \\ a \\	0
	\end{pmatrix} \qquad 
	f_3(e_3) = 
	\begin{pmatrix}
		0 \\
		f_5(e_2)
	\end{pmatrix}
	\]
	for some $a\in \K$ (we used $(3)$). 
	
	\noindent $(6)$ The square involving $f_5$ and $f_6$ gives us that $\phi_{2,1}f_5(e_i) =f_6(\phi_{2,1}(e_i))$ that means
	\[f_6(e_1) = 
	\begin{pmatrix}
		f_5(e_1) \\ 0
	\end{pmatrix} = \begin{pmatrix}
		a \\ 0 \\ 0
	\end{pmatrix} \qquad 
	f_6(e_2) = 
	\begin{pmatrix}
		f_5(e_2) \\0 
	\end{pmatrix} = \begin{pmatrix}
	0 \\ f_7(e_1) 
	\end{pmatrix} =\begin{pmatrix}
	0 \\ b \\	0
	\end{pmatrix} 
	\]
	for some $b\in \K$ (we used $(4)$).
	
	\noindent $(7)$ The square involving $f_4$ and $f_6$ gives us that $\alpha f_6(e_i) =f_4(\alpha(e_i))$ that means
	\[f_4(e_1) = \alpha f_6(e_1)=
	\begin{pmatrix}
		a \\ 0
	\end{pmatrix} \qquad
	f_4(e_2) = \alpha f_6(e_3)=
	\begin{pmatrix}
		0 \\ f_8
	\end{pmatrix}	 
	\]
	\[\begin{pmatrix}
		a \\ f_8
	\end{pmatrix}=f_4(e_1+e_2) = \alpha f_6(e_2)=
	\begin{pmatrix}
		b \\ b
	\end{pmatrix}\]
	so we conclude that $a=b=f_8$. 
	
	\noindent $(8)$ The square involving $f_3$ and $f_4$ gives us that $\beta_{\lambda'}f_3(e_i) =f_4(\beta_\lambda(e_i))$. For $i=1$ we have that
	\[\begin{pmatrix}
		\lambda f_8 \\ f_8
	\end{pmatrix}=f_4(\lambda e_1+e_2) = \beta_{\lambda'}f_3(e_1)=
	\begin{pmatrix}
		\lambda' f_1 \\ f_1
	\end{pmatrix}\]
	so we conclude that $f_1=f_8$ and $\lambda f_1= \lambda' f_1$. 
	
	Therefore we get 
	\begin{equation}\tag{*}\label{eq:R3maps}
		f_1=f_8 \qquad f_2= f_4=f_5=f_7 =f_1 \boldsymbol{1}_2 \qquad  f_3=f_6= f_1 \boldsymbol{1}_3
	\end{equation}
	
	Now if $\lambda=\lambda'$ we conclude that any morphism $f\colon B(\lambda)_{\qR_3} \to B(\lambda)_{\qR_3}$ depends only on $f_1 \in \K$ by \eqref{eq:R3maps}, therefore $\End(B(\lambda)_{\qR_3})\cong \K$, so  $B(\lambda)_{\qR_3}$ is a brick for any $\lambda\in \K$.
	
	If $\lambda\neq\lambda'$ we have that $\lambda f_1= \lambda' f_1$ and if $f_1\neq 0$ it would mean that $\lambda =\lambda'$, contradiction. Thus, for any $\lambda\neq\lambda'$, we must have $\Hom(B(\lambda)_{\qR_3},B(\lambda')_{\qR_3})=0$, so $B(\lambda)_{\qR_3}$ and $B(\lambda')_{\qR_3}$ cannot be isomorphic.
	
	We conclude that $\{B(\lambda)_{\qR_3}\}_{\lambda\in \K}$ is an infinite family of non-isomorphic bricks with the same dimension and consequently the incidence $\K$-algebra $\incidencealgebra{\qR_3}$ is brick-continuous.
\end{proof}

\Paragraphtitlebf{Construction of $G_{\qR_3}$} We have the brick-continuous family $\{B(\lambda)_{\qR_3}\}_{\lambda\in \K}$ and we define the representation $G_{\qR_3}$ substituting $\K^n$ with $\K(x)^n$ (for any $n$) and $\lambda$ with $x$:
\[\begin{tikzcd}[ampersand replacement=\&,sep=small]
	\&\&\& {\K(x)^2} \&\&\& \\
	{\K(x)} \& {\K(x)^2} \& {\K(x)^3} \&\& {\K(x)^3} \& {\K(x)^2} \& {\K(x)} \\
	\&\&\& {\K(x)^2}
	\arrow["{\phi_{1,1}}", from=2-1, to=2-2]
	\arrow["{\phi_{2,1}}", from=2-2, to=2-3]
	\arrow["{\beta_x}", color={rgb,255:red,255;green,51;blue,54}, from=2-3, to=1-4]
	\arrow["\alpha"', from=2-5, to=1-4]
	\arrow["{\psi_{2,1}}"', from=2-6, to=2-5]
	\arrow["{\psi_{1,1}}"', from=2-7, to=2-6]
	\arrow[dashed, no head, from=3-4, to=1-4]
	\arrow["{\psi_{2,1}}", from=3-4, to=2-3]
	\arrow["{\phi_{2,1}}"', from=3-4, to=2-5]
\end{tikzcd}\]
with $\beta_x$ essentially being the previos map $\beta_\lambda$ in which we substitute the element $\lambda\in \K$ with the variable $x$.

\begin{lemma}
	$G_{\qR_3}$ is a generic brick over $\incidencealgebra{\qR_3}$.
\end{lemma}
\begin{proof}
	By \cite[Ch.III, Corollary 3.6]{sim:vol1}, we have that 
	\[\ell_{\incidencealgebra{\qR_3}}(G_{\qR_3}) = \dim_\K(G_{\qR_3}) = \sum_{i\in (\HasseQuiver{\qR_3})_0} \dim_\K((G_{\qR_3})_i) = \infty\]
	So $G_{\qR_3}$ has infinite length. Moreover, following the same procedure as in Lemma \ref{lemma:brickcontR3}, we see that all the components of an endomorphism $f\colon G_{\qR_3}\to G_{\qR_3}$ depends only on the map $f_1\colon \K(x) \to \K(x)$ that is $\K$-linear by definition and satisfies $xf_1(1)=f_1(x)$ (as done previously, substituting $\lambda$ with $x$). We also have that
	\[f_1(1)=f_1(xx^{-1})=xf_1(x^{-1}) \implies x^{-1}f_1(1)=f_1(x^{-1}) \]
	Therefore $f_1$ is completely determined by the element $f_1(1)\in \K(x)$, thus $\End(G_{\qR_3})\cong \K(x)$. Using again \cite[Ch.III, Corollary 3.6]{sim:vol1}
	\[\ell_{\K(x)\op}(G_{\qR_3}) = \dim_{\K(x)}(G_{\qR_3}) = \sum_{i\in (\HasseQuiver{\qR_3})_0} \dim_{\K(x)}((G_{\qR_3})_i) < \infty\]
	So $G_{\qR_3}$ has finite endolength too. By Remark \ref{rem:genmodulek(x)} we conclude that $G_{\qR_3}$ is a generic brick.
\end{proof}
	
	This concludes the construction for the poset $\qR_3$. The other cases are analogous.
\end{proof}

	\subsection{Extension functors}\label{section:extensionfunctor}
	We want to generalise this result to any representation-infinite poset. We recall that a given representation-infinite poset $P$ can be reduced to a critical poset $C$ (by Lemma \ref{lemma:infinitereduction}). We just proved that $\incidencealgebra{C}$ is brick-continuous and we constructed a generic brick for it. If we can ``reverse'' the reduction process we may be able to compute a generic brick for all the representation-infinite posets. For this purpose, we will now define the ``extension'' functors of the contraction and of the subtraction and prove that they have the necessary properties.

\subsubsection{Extension functor of contraction}\label{subsection:extfuncontr}

We recall that any contraction of finite posets can be written as a composite of elementary contractions, so we will focus our attention on elementary contractions.

Let $P$, $P'$ be finite posets with $c\colon P \to P'$ an elementary contraction. If $\card{P}=n$ then $\card{P'}=n-1$ (since $c$ is elementary), so, up to re-indexing the elements of the posets, we may suppose that $P=\{x_1,\dots,x_{n-1},x_n\}$ with $x_{n-1}\leq x_n$ neighbours and $P'=\{y_1,\dots,y_{n-1}\}$ and $c$ defined by $x_i \mapsto y_i$ for $i=1,\dots, n-2$ and $x_{n-1},x_n \mapsto y_{n-1}$.
Let $(\HasseQuiver{P}, I_{P}), (\HasseQuiver{P'}, I_{P'})$ be the Hasse quivers with relations of $P'$ and $P'$ respectively.

Consider the natural map of right projective $\incidencealgebra{P}$-modules $\varphi\colon e_{n-1}(\incidencealgebra{P}) \to e_n(\incidencealgebra{P})$.
We have a fully faithful functor $G\defeq g_*\colon \Mod{(\incidencealgebra{P})_\varphi} \hookrightarrow \Mod{\incidencealgebra{P}}$ given by the universal localisation $g \colon \incidencealgebra{P}\to (\incidencealgebra{P})_\varphi$. Moreover we know that (see \cite[Theorem 4.7]{schofield1985representation})
\[
\Mod{(\incidencealgebra{P})_\varphi} \cong \{M\in \Rep{\HasseQuiver{P}, I_{P}} \sth \text{the map } M_{n-1}\to M_n \text{ is an isomorphism}\}
\]
We get an equivalence $F\colon \Mod{\incidencealgebra{P'}} \xrightarrow{\cong} \Mod{(\incidencealgebra{P})_\varphi}$ in the following way: for a representation $M\in \Rep{\HasseQuiver{P'}, I_{P'}}\cong \Mod{\incidencealgebra{P'}}$ we define 
\[
F(M)_i \defeq 
\begin{cases}
	M_i &\text{for } i=1,\dots, n-1 \\
	M_{n-1} &\text{for } i=n
\end{cases}
\]
and we set $M_{n-1} = F(M)_{n-1} \xrightarrow{\text{id}_{M_{n-1}}} F(M)_n = M_{n-1}$. The other $\K$-linear maps in $F(M)$ are defined from those in $M$ in order to satisfy the relations in $I_P$. For example
\[\begin{tikzcd}[ampersand replacement=\&,cramped,sep=small]
	\& P \&\&\& {P'} \&\& {\Mod{\incidencealgebra{P'}}} \&\& {\Mod{(\incidencealgebra{P})_\varphi}} \& \\
	\& 4 \&\&\& 3 \&\& {V_3} \&\& {V_3} \\
	2 \&\& 3 \&\& 2 \&\& {V_2} \& {V_2} \&\& {V_3} \\
	\& 1 \&\&\& 1 \&\& {V_1} \&\& {V_1}
	\arrow["c", two heads, from=1-2, to=1-5]
	\arrow["F", from=1-7, to=1-9]
	\arrow[from=3-1, to=2-2]
	\arrow[from=3-3, to=2-2]
	\arrow[from=3-5, to=2-5]
	\arrow["{\psi_{23}}"', from=3-7, to=2-7]
	\arrow["{\psi_{23}}", from=3-8, to=2-9]
	\arrow["1"', from=3-10, to=2-9]
	\arrow[dashed, no head, from=4-2, to=2-2]
	\arrow[from=4-2, to=3-1]
	\arrow[from=4-2, to=3-3]
	\arrow[from=4-5, to=3-5]
	\arrow["{\psi_{12}}"', from=4-7, to=3-7]
	\arrow[dashed, no head, from=4-9, to=2-9]
	\arrow["{\psi_{12}}", from=4-9, to=3-8]
	\arrow["{\psi_{23}\psi_{12}}"', from=4-9, to=3-10]
\end{tikzcd}\]

Using the previous description of $\Mod{(\incidencealgebra{P})_\varphi}$, it can be easily seen that $F$ is a well-defined fully faithful dense functor, giving us an equivalence.

\begin{definition}
	We define the \textbf{extension functor} $E_c$ of the elementary contraction $c$ as the fully faithful functor $E_c \defeq G\circ F \colon \Mod{\incidencealgebra{P'}} \xrightarrow{\cong} \Mod{(\incidencealgebra{P})_\varphi} \hookrightarrow \Mod{\incidencealgebra{P}}$.
\end{definition}

By the definition of the functors $F$ and $G$, it is clear that the restriction $E_{c \downarrow} \colon \Modf{\incidencealgebra{P'}} \hookrightarrow \Modf{\incidencealgebra{P}}$ is also a well-defined fully faithful functor.


\subsubsection{Extension functors of subtraction}
We recall that given a subtraction of a poset we may suppose to remove one vertex at the time (c.f. Definition \ref{def:reduction}), so we will focus our attention on subtractions that remove just one element.

Let $P$ be a finite poset, $P'$ a subposet of $P$ with $n=\card{P}=\card{P'}+1$ and denote by $s\colon P' \to P$ the inclusion of $P'$ in $P$ as a subposet. Up to re-indexing the elements of the posets we may suppose that $P=\{1,\dots,n-1, n\}$ and $P'=\{1,\dots,n-1\}$, so the subtraction removed the vertex $n$.
Let $(\HasseQuiver{P}, I_{P}), (\HasseQuiver{P'}, I_{P'})$ be the bound Hasse quivers of $P$ and $P'$ respectively. It is easy to check that $\incidencealgebra{P'} \cong e(\incidencealgebra{P})e$ where $e=\sum_{i=1}^{n-1} e_i$ with $e_i$ the idempotent corresponding to the vertex $i$. Therefore we have the following $\K$-linear covariant functors, with their restrictions to finite-dimensional modules (c.f.~\cite[Ch. I.6]{sim:vol1})
\[\begin{tikzcd}[ampersand replacement=\&,cramped]
	{\Mod{\incidencealgebra{P}}} \&\& {\Mod{\incidencealgebra{P'}}} \&\& {\Modf{\incidencealgebra{P}}} \&\& {\Modf{\incidencealgebra{P'}}}
	\arrow["{\text{res}_e}", from=1-1, to=1-3]
	\arrow["{T_e}"', shift left=3, curve={height=24pt}, between={0.1}{0.9}, from=1-3, to=1-1]
	\arrow["{L_e}", shift right=3, curve={height=-24pt}, between={0.1}{0.9}, from=1-3, to=1-1]
	\arrow["{\text{res}_e}", from=1-5, to=1-7]
	\arrow["{T_e}"', shift left=3, curve={height=24pt}, between={0.1}{0.9}, from=1-7, to=1-5]
	\arrow["{L_e}", shift right=3, curve={height=-24pt}, between={0.1}{0.9}, from=1-7, to=1-5]
\end{tikzcd}\]

where $\text{res}_e \defeq e(\mhyphen)$ is the \textbf{restriction functor}, and $T_e \defeq \incidencealgebra{P}e \otimes_{\incidencealgebra{P'}} \mhyphen$, $L_e \defeq \Hom_{\incidencealgebra{P'}}(e\incidencealgebra{P}, \mhyphen)$ are the \textbf{idempotent embedding functors}.

\begin{lemma}[{c.f.~\cite[Ch.I, Theorem 6.8]{sim:vol1}}]\label{lemma:ffsubtraction}
	We have the following:
	\begin{enumerate}
		\item $T_e$ and $L_e$ are fully faithful;
		\item $\text{res}_e T_e \cong 1_{\Mod{\incidencealgebra{P'}}} \cong \text{res}_e L_e$ (resp.~$\text{res}_e T_e \cong 1_{\Modf{\incidencealgebra{P'}}} \cong \text{res}_e L_e$);
		\item $T_e \dashv \text{res}_e \dashv L_e$.
	\end{enumerate}
\end{lemma}

Working with representations of the respective bound quivers it is easy to check that, for $M\in \Rep{\HasseQuiver{P'}, I_{P'}} \cong \Mod{\incidencealgebra{P'}}$, we have 
\[\begin{aligned}
	L_e(M)_i &\cong M_i  \quad \forall i=1,\dots, n-1 \quad \text{and}\quad L_e(M)_n \cong \lim_{n<_P j} M_j 
	\\
	T_e(M)_i &\cong M_i  \quad \forall i=1,\dots, n-1 \quad \text{and}\quad T_e(M)_n \cong \colim\limits_{j<_P n } M_j
\end{aligned}
\]
For example, if $M\in \Rep{\HasseQuiver{P'}, I_{P'}}$ and $\widetilde{M} \in \Mod{\incidencealgebra{P'}}$ is the corresponding module we have
\[T_e(M)_i \cong e_i T_e(\widetilde{M}) = e_i \incidencealgebra{P}e \otimes_{\incidencealgebra{P'}} \widetilde{M} \cong e_i \incidencealgebra{P'} \otimes_{\incidencealgebra{P'}} \widetilde{M} \cong e_i \widetilde{M} \cong M_i \qquad \forall i=1,\dots, n-1\]
while $T_e(M)_n \cong e_n T_e(\widetilde{M}) = e_n \incidencealgebra{P}e \otimes_{\incidencealgebra{P'}} \widetilde{M}$. 

We have canonical maps $T_e(M)_j = e_j \incidencealgebra{P}e \otimes_{\incidencealgebra{P'}} \widetilde{M} \xrightarrow{p_{j\to n}\cdot} e_n \incidencealgebra{P}e \otimes_{\incidencealgebra{P'}} \widetilde{M} = T_e(M)_n$ for all $j<_P n$, given by left multiplication with the path from $j$ to $n$ (we have such a path since $j<_P n$ and it is well-defined since all such paths are equal thanks to the relations in the incidence algebra).

Given $e_n p e\otimes m \in T_e(M)_n$ with $p$ a path $\ell \rightsquigarrow n$, we have that $e_n p e\otimes m = p_{\ell\to n}\cdot e_\ell 1_{\incidencealgebra{P}} e \otimes m$ (again, due to the relations in the incidence algebra). With this information, it is easy to check the universal property of the colimit:

\noindent \begin{minipage}{0.55\textwidth}
	given the diagram on the right, we can define (with the previous notation)
	\[g(e_n p e\otimes m) \defeq h_\ell(e_\ell 1_{\incidencealgebra{P}} e\otimes m) \]
	Extending it by linearity, we get the unique map $g\colon T_e(M)_n \to C$ such that $g(p_{j\to n}\cdot (e_j a e)\otimes m) = h_j(e_j a e\otimes m)$ for all $j<_P n$ (thanks to its definition).	
\end{minipage}
\hfill
\begin{minipage}{0.4\textwidth}
	\[\begin{tikzcd}[ampersand replacement=\&,cramped]
		{T_e(M)_j} \&\& {T_e(M)_k} \\
		\& {T_e(M)_n} \\
		\& C
		\arrow["{p_{j\to k} \cdot}", dotted, from=1-1, to=1-3]
		\arrow["{p_{j\to n}\cdot}"', from=1-1, to=2-2]
		\arrow["{h_j}"', curve={height=18pt}, from=1-1, to=3-2]
		\arrow["{p_{k\to n}\cdot}", from=1-3, to=2-2]
		\arrow["{h_k}", curve={height=-18pt}, from=1-3, to=3-2]
		\arrow["{\exists! g}", dotted, from=2-2, to=3-2]
	\end{tikzcd}\]
\end{minipage}

The proof for the functor $L_e$ has a similar procedure.

\begin{definition}
	With the previous setting, we define the \textbf{extension functors} $E^{\Hom}_s$, $E^\otimes_s$ of the subtraction given by $s\colon P'\to P$ as the functors $E^{\Hom}_s\defeq L_e$, $E^\otimes_s\defeq T_e$.
\end{definition}

In the following, we will simply denote them as $E_s$ when we are not be interested in which particular functor we are applying but only on their common properties. Moreover $E_{s \downarrow}$ will denote the restriction to finite-dimensional modules.

	\subsection{Generic bricks of representation-infinite posets}\label{section:genbrickreprinf}
	We will now prove that, given any representation-infinite poset $P$, we can compute a brick-continuous family and a generic brick in $\Mod{\incidencealgebra{P}}$ using the extension functors we just defined.

\Paragraphtitlebf{Reduction of $P$.} Let $P$ be any representation-infinite poset. By Lemma \ref{lemma:infinitereduction}, it can be reduced to a critical poset $C$ (i.e.~a poset with minimally representation-infinite incidence algebra). As done before, we may assume that any step of the reduction process removes only one vertex. Let $\card{P}=n$ and $\card{C}=n-m$, so the reduction process involved $m$ points of $P$ through a sequence of subtractions and/or contractions. We obtain a list of posets $P_i$ such that $\card{P_i}=n-i$ for $0\leq i\leq m$ with $P=P_0$ and $C=P_m$ and we denote by $r_j$ the reduction applied at the $j$-th step ($1\leq j\leq m$), that is 
\[r_j=c\colon P_{j-1} \to P_j\]
if we applied an elementary contraction and
\[r_j=s\colon P_j \to P_{j-1}\]
if we applied a subtraction of a vertex.
From this list of posets and reduction maps we can define a composition of the corresponding extension functors as
\[\begin{tikzcd}[ampersand replacement=\&,cramped]
	{\Mod{\incidencealgebra{C}}} \& {\Mod{\incidencealgebra{P_{m-1}}}} \& \cdots \& {\Mod{\incidencealgebra{P_1}}} \& {\Mod{\incidencealgebra{P}}}
	\arrow["{E_{r_m}}", from=1-1, to=1-2]
	\arrow["{E_{r_{m-1}}}", from=1-2, to=1-3]
	\arrow["{E_{r_2}}", from=1-3, to=1-4]
	\arrow["{E_{r_1}}", from=1-4, to=1-5]
\end{tikzcd}\]
that we denote by $\widetilde{E}\defeq E_{r_1}E_{r_2}\cdots E_{r_{m-1}}E_{r_m}$. We will denote by $\widetilde{E}_\downarrow \defeq E_{r_1 \downarrow} E_{r_2\downarrow} \cdots E_{r_{m-1} \downarrow} E_{r_m \downarrow}$ its restriction to finite-dimensional modules.

Let $\{B(\lambda)_C\}_{\lambda\in\K}$ be the brick-continuous family and $G_C$ be the generic brick of the critical poset $C$ given by Theorem \ref{thm:genbrickcritical} (see Appendix \ref{appendix} for the list). By the construction of $G_C$ we have that $\dim_\K((G_C)_i)=\infty$ and $\dim_{\K(x)}((G_C)_i) = (y_C)_i$ for all $i\in (\HasseQuiver{C})_0$ since $(G_C)_i=\K(x)^{(y_C)_i}$ with $y_C$ the minimal sincere positive radical vector of the quadratic form $q_C$ (see Theorem \ref{thm:genbrickcritical}).

Then, with the previous setting, we have the following results:

\begin{lemma}\label{lemma:extendinggenbricks}
	$G_P \defeq \widetilde{E}(G_C)= E_{r_1}E_{r_2}\cdots E_{r_{m-1}}E_{r_m}(G_C)$ is a generic brick for $\incidencealgebra{P}$ with $(G_P)_i =\K(x)^{h_i}$ for all $i\in (\HasseQuiver{P})_0$ for some $h_i >0$ and $\End(G_P)\cong \K(x)$.
\end{lemma}
\begin{proof}
	We prove the statement by induction on the number $m$ of reduction steps.
	Let $m=1$, then we just removed a vertex through a contraction or a subtraction and, up to re-indexing, let $P=\{1,\dots, n-1, n\}$ and $C=\{1,\dots, n-1\}$.
	
	$(i)$ If we applied a contraction $c\colon P\to C$, contracting $n-1$ and $n$, we have $\widetilde{E}=E_c$. Now $(G_P)_i = (G_C)_i$ for the vertices $1,\dots, n-1$ and  $(G_P)_n=(G_C)_{n-1}$ in vertex $n$ by definition. So we get that
	\[\dim_\K(G_P)=\dim_\K(G_C) + \dim_\K((G_C)_{n-1})\]
	Since $G_C$ is a generic brick we have (by \cite[Ch.III, Corollary 3.6]{sim:vol1})
	\[\infty=\ell_{\incidencealgebra{C}}(G_C) =\dim_\K(G_C)\]
	Therefore, we conclude that 
	\[\ell_{\incidencealgebra{P}}(G_P) =\dim_\K(G_P)=\dim_\K(G_C) + \dim_\K((G_C)_{n-1})=\infty\]
	Now $E_c$ is fully faithful (by its definition), so $ \End(G_P)=\End(E_c(G_C)) \cong \End(G_C) \cong \K(x)$. Again by \cite[Ch.III, Corollary 3.6]{sim:vol1}, since $G_C$ has finite endolength we also get 
	\[\begin{aligned}
		\ell_{\K(x)\op}(G_P) &=\dim_{\K(x)}(G_P)=\dim_{\K(x)}(G_C) + \dim_{\K(x)}((G_C)_{n-1}) \\
		&=\ell_{\K(x)\op}(G_C) + (y_C)_{n-1} < \infty
	\end{aligned}\]
	Therefore $G_P$ has infinite length but finite endolength, so it is a generic module and by Remark \ref{rem:genmodulek(x)} it is a generic brick. We also conclude that $(G_P)_i=(E_c(G_C))_i =\K(x)^{h_i}$ for all $i\in (\HasseQuiver{P})_0$ for some $h_i >0$ by definition of $E_c$.
	
	$(ii)$ If we applied a subtraction $s\colon C\to P$, removing the vertex $n$, we have $\widetilde{E}=E_s$. Suppose $\widetilde{E} = E_s = E^{\Hom}_s $ (the case with $E^\otimes_s$ is analogous). Similarly to $(i)$, by definition of $E_s$ we have
	
	\[
	\begin{aligned}
		\dim_\K(G_P)&=\dim_\K(G_C) + \dim_\K(\lim_{n<_P j} G(C)_j) \\
		\dim_{\K(x)}(G_P)&=\dim_{\K(x)}(G_C) + \dim_{\K(x)}(\lim_{n<_P j} G(C)_j)
	\end{aligned}
	\]
	
	\noindent and $\lim_{n<_P j} G(C)_j \leq \prod_{n<_P j} (G_C)_j = \prod_{n<_P j} \K(x)^{(y_C)_j}$ (finite product). Moreover $E_s$ is fully faithful by Lemma \ref{lemma:ffsubtraction}. Therefore, with a similar reasoning as in $(i)$, we conclude that $\ell_{\incidencealgebra{P}}(G_P)=\infty$, $\End(G_P) \cong \End(G_C) \cong \K(x)$ and $\ell_{\K(x)\op}(G_P)<\infty$. Thus, $G_P$ has infinite length but finite endolength, so it is a generic module and by Remark \ref{rem:genmodulek(x)} it is a generic brick. By definition of $E_s$ we also conclude that $(G_P)_i=(E_s(G_C))_i =\K(x)^{h_i}$ for all $i\in (\HasseQuiver{P})_0$ for some $h_i >0$.
	
	Now suppose the result holds for $m-1$ reduction steps, so $\widetilde{G}= E_{r_2}\cdots E_{r_{m-1}}E_{r_m}(G_C)$ is a generic brick for $\incidencealgebra{P_1}$ with $(\widetilde{G})_i =\K(x)^{\tilde{h}_i}$ for all $i\in (\HasseQuiver{P})_0$ for some $\tilde{h}_i >0$ and $\End(\widetilde{G})\cong \K(x)$. We have that $r_1$ can either be a contraction or a subtraction but in any case, since $\widetilde{G}$ has a similar behavior as $G_C$, we can apply the same reasoning used in $(i)$ and $(ii)$ to conclude that $G_P=\widetilde{E}(G_C)= E_{r_1}(\widetilde{G})$ is a generic brick for $\incidencealgebra{P}$ with $(G_P)_i =\K(x)^{h_i}$ for all $i\in (\HasseQuiver{P})_0$ for some $h_i >0$ and $\End(G_P)\cong \K(x)$. So we conclude.
\end{proof}

\begin{lemma}\label{lemma:extendingbrickcont}
	The family $\widetilde{E}_\downarrow (B(\lambda)_C) = E_{r_1 \downarrow}E_{r_2\downarrow}\cdots E_{r_{m-1} \downarrow}E_{r_m \downarrow} (B(\lambda)_C)$, for $\lambda\in \K$, contains a brick-continuous family for $\incidencealgebra{P}$ having a fixed dimension vector.	
\end{lemma}
\begin{proof}
	As previously done, we prove the statement by induction on the number $m$ of reduction steps.
	Let $m=1$, then we just removed a vertex through a contraction or a subtraction and, up to re-indexing, let $P=\{1,\dots, n-1, n\}$ and $C=\{1,\dots, n-1\}$.
	
	Let $\underline{d} = (d_1, \dots, d_{n-1}) \defeq y_C = \dimvec{B(\lambda)_C}$ and $d \defeq \dim_\K(B(\lambda)_C) = \sum_{i=1}^{n-1} d_i$.	
	By fully faithfulness we have that $\{\widetilde{E}_\downarrow (B(\lambda)_C)\}_{\lambda\in \K}$ is always an infinite family of non-isomorphic bricks.
	
	$(i)$ If we applied a contraction $c\colon P\to C$, contracting $n-1$ and $n$, we have $\widetilde{E}_\downarrow=E_{c\downarrow}$.
	Then $\underline{d}' \defeq \dimvec{E_{c\downarrow}(B(\lambda)_C)}=(\underline{d}, d_{n-1})$ (and $\dim_\K(E_{c\downarrow}(B(\lambda)_C))= d + d_{n-1}$) for any $\lambda\in\K$. So, for $\lambda\in \K$, $B(\lambda)_P\defeq E_{c\downarrow}(B(\lambda)_C)$ gives a brick-continuous family with dimension vector $\underline{d}'$.
	
	$(ii)$ If we applied a subtraction $s\colon C\to P$, removing the vertex $n$, we have $\widetilde{E}_\downarrow=E_{s\downarrow}$. Suppose $\widetilde{E}_\downarrow = E_{s\downarrow} = E^{\Hom}_{s\downarrow}$ (the case with $E^\otimes_{s\downarrow}$ is analogous). 
	
	Then $\dimvec{E_{s\downarrow}(B(\lambda)_C)}=(\underline{d}, \dim_\K(\lim_{n<_P j} (B(\lambda)_C)_j))$. But now we have that
	\[
	0 \leq \dim_\K(\lim_{n<_P j} (B(\lambda)_C)_j) \leq D_L \defeq \dim_\K(\prod_{n<_P j} (B(\lambda)_C)_j) = \sum_{n<_P j} d_j
	\]
	Since we have infinite choices of $\lambda\in \K$ but only finitely many possibilities for $\dim_\K(\lim_{n<_P j} (B(\lambda)_C)_j)$, by the \emph{Infinite Pigeonhole principle} there must be a positive integer $d_L\in\Z$, with $0\leq d_L \leq D_L$, and an infinite subset $I\subseteq\K$ such that $\dim_\K(\lim_{n<_P j} (B(\lambda_i)_C)_j) = d_L$ for any $\lambda_i \in I$. Thus $\{B(\lambda_i)_P\}_{\lambda_i\in I} \defeq \{E_{s\downarrow}(B(\lambda_i)_C)\}_{\lambda_i\in I}$ is a brick-continuous family with dimension vector $\underline{d}' \defeq (\underline{d}, d_L)$ (and dimension $d+ d_L$).
	
	If we suppose that the result holds for $m-1$ reduction steps, with an analogous reasoning we conclude by induction that $\{\widetilde{E}_\downarrow (B(\lambda)_C)\}_{\lambda\in \K}$ contains a brick continuous family $\{B(\lambda_i)_P\}_{\lambda_i\in I}$ ($I\subseteq \K$ infinite subset) with a fixed dimension vector.	
\end{proof}

We can summarise our discussion with the following theorem:

\begin{theorem}\label{thm:genbrickinfinite}
	Any representation-infinite poset $P$ is brick-continuous and admits a generic brick over $\incidencealgebra{P}$ that can be explicitly computed.
\end{theorem}

\begin{example}
	The poset $P$ can be reduced to the critical poset $\qDtilde_4$ using a contraction		
	
	
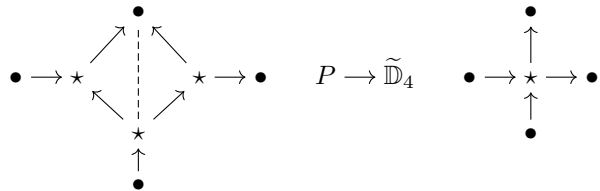
\[\begin{tikzcd}[ampersand replacement=\&,cramped,sep=scriptsize]
		\&\& \bullet \&\&\&\&\&\& \bullet \& \\
		\bullet \& \star \&\& \star \& \bullet \& P \& {\qDtilde_4} \& \bullet \& \star \& \bullet \\
		\&\& \star \&\&\&\&\&\& \bullet \\
		\&\& \bullet
		\arrow[from=2-1, to=2-2]
		\arrow[from=2-2, to=1-3]
		\arrow[from=2-4, to=1-3]
		\arrow[from=2-4, to=2-5]
		\arrow[from=2-6, to=2-7]
		\arrow[from=2-8, to=2-9]
		\arrow[from=2-9, to=1-9]
		\arrow[from=2-9, to=2-10]
		\arrow[dashed, no head, from=3-3, to=1-3]
		\arrow[from=3-3, to=2-2]
		\arrow[from=3-3, to=2-4]
		\arrow[from=3-9, to=2-9]
		\arrow[from=4-3, to=3-3]
	\end{tikzcd}\]
	
	\noindent so it is representation-infinite by Theorem \ref{thm:Lou75concealed}. Thus, we can apply Theorem \ref{thm:genbrickinfinite} to obtain a generic brick for the poset $P$:
	
	\[\begin{tikzcd}[ampersand replacement=\&,cramped,column sep=small,row sep=scriptsize]
		\& {K(x)} \&\&\&\&\&\& {K(x)} \&\& \\
		{K(x)} \& {K(x)^2} \& {K(x)} \& {G_{\qDtilde_4}} \& {\widetilde{E}(G_{\qDtilde_4})} \& {K(x)} \& {K(x)^2} \&\& {K(x)^2} \& {K(x)} \\
		\& {K(x)} \&\&\&\&\&\& {K(x)^2} \\
		\&\&\&\&\&\&\& {K(x)}
		\arrow["\begin{array}{c} \left(\begin{smallmatrix} 0 \\ 1 \end{smallmatrix} \right) \end{array}", from=2-1, to=2-2]
		\arrow["{\left(\begin{smallmatrix} 1 & 1 \end{smallmatrix} \right)}"', from=2-2, to=1-2]
		\arrow["{\left(\begin{smallmatrix} 1 & x \end{smallmatrix} \right)}"', from=2-2, to=2-3]
		\arrow["{\widetilde{E}}", maps to, from=2-4, to=2-5]
		\arrow["\begin{array}{c} \left(\begin{smallmatrix} 0 \\ 1 \end{smallmatrix} \right) \end{array}", from=2-6, to=2-7]
		\arrow["{\left(\begin{smallmatrix} 1 & 1 \end{smallmatrix} \right)}", from=2-7, to=1-8]
		\arrow["{\left(\begin{smallmatrix} 1 & 1 \end{smallmatrix} \right)}"', from=2-9, to=1-8]
		\arrow["{\left(\begin{smallmatrix} 1 & x \end{smallmatrix} \right)}", from=2-9, to=2-10]
		\arrow["\begin{array}{c} \left(\begin{smallmatrix} 1 \\ 0 \end{smallmatrix} \right) \end{array}", from=3-2, to=2-2]
		\arrow[dashed, no head, from=3-8, to=1-8]
		\arrow["1", from=3-8, to=2-7]
		\arrow["1"', from=3-8, to=2-9]
		\arrow["\begin{array}{c} \left(\begin{smallmatrix} 1 \\ 0 \end{smallmatrix} \right) \end{array}", from=4-8, to=3-8]
	\end{tikzcd}\]
	
\end{example}

Applying ideas similar to those used previously, we can easily prove the following:

\begin{theorem}\label{thm:generalisation}
	Let $\Lambda, \Lambda'$ be finite-dimensional algebras.
	\begin{enumerate}
		\item If $\Lambda'$ is brick-continuous and $F \colon \Modf{\Lambda'} \to \Modf{\Lambda}$ is a fully faithful exact functor, then $\Lambda$ is brick-continuous.
		\item If $\Lambda'$ has a generic brick and $F \colon \Mod{\Lambda'} \to \Mod{\Lambda}$ is a fully faithful exact functor that preserves endofinite modules, then $\Lambda$ admits a generic brick.
	\end{enumerate}
\end{theorem}
\begin{proof}
	$(1)$ By Lemma \ref{lemma:brickcontdef} we can suppose to have a brick-continuous family $\{B_i\}_{i\in I}$ of $\Lambda'$ with the same dimension vector. Then $\{F(B_i)\}_{i\in I}$ is a family of non-isomorphic bricks (by fully faithfulness) with the same dimension vector (since exactness gives us a group homomorphism of the corresponding Grothendieck groups \cite[Ch.II, 6.1.5]{weibel2013k}).
	
	$(2)$ Let $G$ be a generic brick of $\Lambda'$. Then $F(G)$ is a brick (by fully faithfulness) of infinite length (by fully faithfulness and exactness) with finite endolength (since $F$ preserves endofinite modules), so a generic brick for $\Lambda$.
\end{proof}

\begin{remark}
	Theorem \ref{thm:generalisation} could be used to study properties of algebras \emph{linked} to incidence algebras:
	\begin{itemize}
		\itemsep=0em
		\item If $P$ is a representation-infinite poset and $f\colon\Lambda \to \incidencealgebra{P}$ is a ring epimorphism, then $\Lambda$ is brick-continuous and admits a generic brick. Indeed we can apply the Theorem to $F=f_*\colon \Mod{\incidencealgebra{P}} \to \Mod{\Lambda}$ (and $F=f_*\colon \Modf{\incidencealgebra{P}} \to \Modf{\Lambda}$).
		\item In general if $\Lambda'=\incidencealgebra{P}$ for a representation-infinite poset $P$, then the existence of a functor $F$ as in the Theorem gives us some information about $\Lambda$.
	\end{itemize}
\end{remark}

\begin{example}
	Consider the poset $\qR_3$, the map $f\colon \palgebra{\qR_3}\to \palgebra{\qR_3}/I_{\qR_3}=\incidencealgebra{\qR_3}$ gives a ring epimorphism. Therefore we can obtain a generic brick for the algebra $\palgebra{\qR_3}$ from the generic brick $G_{\qR_3}$, using the previous result:
	\[\begin{tikzcd}[ampersand replacement=\&,cramped,sep=small]
		\&\&\&\&\& {\K(x)^2} \&\&\& \\
		{\Mod{\incidencealgebra{\qR_3}}} \&\& {\K(x)} \& {\K(x)^2} \& {\K(x)^3} \&\& {\K(x)^3} \& {\K(x)^2} \& {\K(x)} \\
		\&\&\&\&\& {\K(x)^2} \\
		\&\&\&\&\& {\K(x)^2} \\
		{\Mod{\palgebra{\qR_3}}} \&\& {\K(x)} \& {\K(x)^2} \& {\K(x)^3} \&\& {\K(x)^3} \& {\K(x)^2} \& {\K(x)} \\
		\&\&\&\&\& {\K(x)^2}
		\arrow["{f_*}"', from=2-1, to=5-1]
		\arrow["{\phi_{1,1}}", from=2-3, to=2-4]
		\arrow["{\phi_{2,1}}", from=2-4, to=2-5]
		\arrow["{\beta_x}", from=2-5, to=1-6]
		\arrow["\alpha"', from=2-7, to=1-6]
		\arrow["{\psi_{2,1}}"', from=2-8, to=2-7]
		\arrow["{\psi_{1,1}}"', from=2-9, to=2-8]
		\arrow[dashed, no head, from=3-6, to=1-6]
		\arrow["{\psi_{2,1}}", from=3-6, to=2-5]
		\arrow["{\phi_{2,1}}"', from=3-6, to=2-7]
		\arrow["{\phi_{1,1}}", from=5-3, to=5-4]
		\arrow["{\phi_{2,1}}", from=5-4, to=5-5]
		\arrow["{\beta_x}", from=5-5, to=4-6]
		\arrow["\alpha"', from=5-7, to=4-6]
		\arrow["{\psi_{2,1}}"', from=5-8, to=5-7]
		\arrow["{\psi_{1,1}}"', from=5-9, to=5-8]
		\arrow["{\psi_{2,1}}", from=6-6, to=5-5]
		\arrow["{\phi_{2,1}}"', from=6-6, to=5-7]
	\end{tikzcd}\]	
	
\end{example}

	\newpage
	
	\appendix 
	
	\section{List of generic bricks of critical posets}\label{appendix}
	
	In this Appendix we will list the generic bricks constructed following the proof of Theorem \ref{thm:genbrickcritical} (choosing a particular orientation of the edges for each family and considering the opposite poset in some cases, namely $\qR_1, \qR_2$ and $\qR_6$). For a critical poset $C$ we denoted by $G_C$ the constructed generic brick. We recall (as in the proof of the Theorem) that the brick-continuous family $B(\lambda)_C$ of $C$ (with $\lambda\in\K$) is obtained from the representation of $G_C$ substituting $\K(x)^n$ with $\K^n$ (for any $n$) in the vertices and the map $\beta_x$ becomes $\beta_\lambda$ (substituting $x$ with $\lambda\in\K$).

Define $\phi_{n,m}\defeq \left( \begin{smallmatrix} \boldsymbol{1}_n \\ \boldsymbol{0}_{m,n}\end{smallmatrix} \right)$, $\psi_{n,m}\defeq \left( \begin{smallmatrix} \boldsymbol{0}_{m,n} \\ \boldsymbol{1}_n \end{smallmatrix} \right)$ (with $\boldsymbol{1}_n \in \matspace{n}$ identity matrix, $\boldsymbol{0}_{m,n}\in \matspace{m\times n}$ zero matrix).

\begin{minipage}{0.45\textwidth}
	\noindent $(\qAtilde_3)$
	\[\begin{tikzcd}[ampersand replacement=\&,cramped,sep=tiny]
		\& {\K(x)} \& \\
		{\K(x)} \&\& {\K(x)} \\
		\& {\K(x)}
		\arrow["x", color={rgb,255:red,255;green,5;blue,5}, from=2-1, to=1-2]
		\arrow["1"', from=2-1, to=3-2]
		\arrow["1"', from=2-3, to=1-2]
		\arrow["1", from=2-3, to=3-2]
	\end{tikzcd}\]
\end{minipage}
\hfill
\begin{minipage}{0.45\textwidth}
	\noindent $(\qDtilde_4)$
	\[\begin{tikzcd}[ampersand replacement=\&,cramped,sep=tiny]
		{\K(x)} \&\& {\K(x)} \\
		\& {\K(x)^2} \\
		{\K(x)} \&\& {\K(x)}
		\arrow["\begin{array}{c} \begin{array}{c} \left( \begin{smallmatrix} 1 \\ 0\end{smallmatrix}\right) \end{array} \end{array}"', from=1-3, to=2-2]
		\arrow["{{\left( \begin{smallmatrix} 1 & 1 \end{smallmatrix}\right)}}", from=2-2, to=1-1]
		\arrow["{{\left( \begin{smallmatrix} x & 1 \end{smallmatrix}\right)}}", color={rgb,255:red,255;green,51;blue,51}, from=2-2, to=3-1]
		\arrow["\begin{array}{c} \begin{array}{c} \left( \begin{smallmatrix} 0 \\ 1\end{smallmatrix}\right) \end{array} \end{array}"', from=3-3, to=2-2]
	\end{tikzcd}\]
\end{minipage}
\smallskip

\begin{minipage}{0.45\textwidth}
	\noindent $(\qEtilde_6)$
	\[\begin{tikzcd}[ampersand replacement=\&,sep=tiny]
		\&\& {\K(x)} \&\& \\
		\&\& {\K(x)^2} \\
		{\K(x)} \& {\K(x)^2} \& {\K(x)^3} \& {\K(x)^2} \& {\K(x)}
		\arrow["{\phi_{1,1}}", from=1-3, to=2-3]
		\arrow["{\beta_x}", color={rgb,255:red,255;green,51;blue,54}, from=2-3, to=3-3]
		\arrow["{\phi_{1,1}}"', from=3-1, to=3-2]
		\arrow["{\phi_{2,1}}"', from=3-2, to=3-3]
		\arrow["{\psi_{2,1}}", from=3-4, to=3-3]
		\arrow["{\psi_{1,1}}", from=3-5, to=3-4]
	\end{tikzcd}\]
\end{minipage}
\hfill
\begin{minipage}{0.45\textwidth}
	$\beta_x\defeq \left( \begin{smallmatrix} x & 1 \\ 1 & 1 \\ 1 & 0 \end{smallmatrix} \right)$
\end{minipage}
\smallskip

\begin{minipage}{0.45\textwidth}
	\noindent $(\qEtilde_7)$
	\[\begin{tikzcd}[ampersand replacement=\&,sep=tiny]
		\&\&\& {\K(x)^2} \&\&\& \\
		{\K(x)} \& {\K(x)^2} \& {\K(x)^3} \& {\K(x)^4} \& {\K(x)^3} \& {\K(x)^2} \& {\K(x)}
		\arrow["{\beta_x}", color={rgb,255:red,255;green,51;blue,54}, from=1-4, to=2-4]
		\arrow["{\phi_{1,1}}", from=2-1, to=2-2]
		\arrow["{\phi_{2,1}}", from=2-2, to=2-3]
		\arrow["{\phi_{3,1}}", from=2-3, to=2-4]
		\arrow["{\psi_{3,1}}", from=2-5, to=2-4]
		\arrow["{\psi_{2,1}}", from=2-6, to=2-5]
		\arrow["{\psi_{1,1}}", from=2-7, to=2-6]
	\end{tikzcd}\]
\end{minipage}
\hfill
\begin{minipage}{0.3\textwidth}
	$\beta_x\defeq \left( \begin{smallmatrix} 1 & x  \\ 1 & 0 \\ 1 & 1 \\ 0 & 1 \end{smallmatrix} \right)$
\end{minipage}
\smallskip

\begin{minipage}{0.45\textwidth}
	\noindent $(\qEtilde_8)$
	\[\begin{tikzcd}[ampersand replacement=\&,sep=tiny]
		\&\& {\K(x)^3} \&\&\&\&\& \\
		{\K(x)^2} \& {\K(x)^4} \& {\K(x)^6} \& {\K(x)^5} \& {\K(x)^4} \& {\K(x)^3} \& {\K(x)^2} \& {\K(x)}
		\arrow["{\beta_x}", color={rgb,255:red,255;green,51;blue,54}, from=1-3, to=2-3]
		\arrow["{\psi_{2,2}}", from=2-1, to=2-2]
		\arrow["{\psi_{4,2}}", from=2-2, to=2-3]
		\arrow["{\phi_{5,1}}", from=2-4, to=2-3]
		\arrow["{\phi_{4,1}}", from=2-5, to=2-4]
		\arrow["{\phi_{3,1}}", from=2-6, to=2-5]
		\arrow["{\phi_{2,1}}", from=2-7, to=2-6]
		\arrow["{\phi_{1,1}}", from=2-8, to=2-7]
	\end{tikzcd}\]
\end{minipage}
\hfill
\begin{minipage}{0.2\textwidth}
	$\beta_x\defeq \left( \begin{smallmatrix} x & 1 & 0 \\ 0 & 0 & 1 \\ 1 & 1 & 0 \\ 1 & 0 & 1 \\ 1 & 1 & 0 \\ 0 & 1 & 0 \end{smallmatrix} \right)$
\end{minipage}
\smallskip

\begin{minipage}{0.45\textwidth}
	\noindent $(\qR_1)$
	\[\begin{tikzcd}[ampersand replacement=\&,sep=tiny]
		\&\&\&\&\& {\K(x)^2} \& \\
		{\K(x)} \& {\K(x)^2} \& {\K(x)^3} \& {\K(x)^4} \& {\K(x)^5} \&\& {\K(x)^3} \\
		\&\&\&\&\& {\K(x)^4} \\
		\&\&\&\&\& {\K(x)^2}
		\arrow["{\phi_{1,1}}", from=2-1, to=2-2]
		\arrow["{\phi_{2,1}}", from=2-2, to=2-3]
		\arrow["{\phi_{3,1}}", from=2-3, to=2-4]
		\arrow["{\phi_{4,1}}", from=2-4, to=2-5]
		\arrow["{\beta_x}", color={rgb,255:red,255;green,61;blue,51}, from=2-5, to=1-6]
		\arrow["{\alpha_2}"', from=2-7, to=1-6]
		\arrow[dashed, no head, from=3-6, to=1-6]
		\arrow["{\psi_{4,1}}", from=3-6, to=2-5]
		\arrow["{\alpha_1}"', from=3-6, to=2-7]
		\arrow["{\psi_{2,2}}", from=4-6, to=3-6]
	\end{tikzcd}\]
\end{minipage}
\hfill
\begin{minipage}{0.3\textwidth}
	$\alpha_1 \defeq \left( \begin{smallmatrix} \boldsymbol{1}_3 & \begin{smallmatrix}
			0 \\ 1 \\ 0
	\end{smallmatrix} \end{smallmatrix} \right)$ 
	\smallskip
	
	$\alpha_2 \defeq \left( \begin{smallmatrix} \boldsymbol{1}_2 & \begin{smallmatrix}
			1 \\ 1
	\end{smallmatrix} \end{smallmatrix} \right)$
	\smallskip
	
	$\beta_x \defeq \left( \begin{smallmatrix} x & 1 & 0 & 1 & 0 \\ 1 & 0 & 1 & 1 & 1
	\end{smallmatrix} \right)$
	
\end{minipage}
\smallskip

\begin{minipage}{0.45\textwidth}
	\noindent $(\qR_2)$
	\[\begin{tikzcd}[ampersand replacement=\&,sep=tiny]
		\&\&\&\& {\K(x)} \&\& \\
		\&\&\& {\K(x)^3} \&\& {\K(x)^4} \& {\K(x)^2} \\
		{\K(x)} \& {\K(x)^2} \& {\K(x)^3} \& {\K(x)^4} \& {\K(x)^5}
		\arrow["{\beta'_x}", color={rgb,255:red,255;green,51;blue,51}, from=2-4, to=1-5]
		\arrow["{\beta_x}"', color={rgb,255:red,255;green,51;blue,51}, from=2-6, to=1-5]
		\arrow["{\psi_{2,2}}"', from=2-7, to=2-6]
		\arrow["{\phi_{1,1}}"', from=3-1, to=3-2]
		\arrow["{\phi_{2,1}}"', from=3-2, to=3-3]
		\arrow["{\phi_{3,1}}"', from=3-3, to=3-4]
		\arrow["{\phi_{4,1}}"', from=3-4, to=3-5]
		\arrow[dashed, no head, from=3-5, to=1-5]
		\arrow["{\alpha_2}"{pos=0.6}, from=3-5, to=2-4]
		\arrow["{\alpha_1}"', from=3-5, to=2-6]
	\end{tikzcd}\]
\end{minipage}
\hfill
\begin{minipage}{0.3\textwidth}
	$\alpha_1 \defeq \left( \begin{smallmatrix} \boldsymbol{1}_4 & \begin{smallmatrix}
			0 \\ 0 \\ 0 \\ 0
	\end{smallmatrix} \end{smallmatrix} \right)$ 
	\smallskip
	
	$\alpha_2 \defeq \left( \begin{smallmatrix} \boldsymbol{1}_3 & \begin{smallmatrix}
			0 & 1 \\ 1 & 1 \\ 0 & 1
	\end{smallmatrix} \end{smallmatrix} \right)$
	\smallskip
	
	$\beta_x \defeq \left( \begin{smallmatrix} 1 & x & (-1-x) & x	\end{smallmatrix} \right)$
	\smallskip 
	
	$\beta'_x \defeq \left( \begin{smallmatrix} 1 & x & (-1-x)	\end{smallmatrix} \right)$
\end{minipage}
\smallskip

\begin{minipage}{0.45\textwidth}
	\noindent $(\qR_3)$
	\[\begin{tikzcd}[ampersand replacement=\&,sep=tiny]
		\&\&\& {\K(x)^2} \&\&\& \\
		{\K(x)} \& {\K(x)^2} \& {\K(x)^3} \&\& {\K(x)^3} \& {\K(x)^2} \& {\K(x)} \\
		\&\&\& {\K(x)^2}
		\arrow["{\phi_{1,1}}", from=2-1, to=2-2]
		\arrow["{\phi_{2,1}}", from=2-2, to=2-3]
		\arrow["{\beta_x}", color={rgb,255:red,255;green,51;blue,54}, from=2-3, to=1-4]
		\arrow["\alpha"', from=2-5, to=1-4]
		\arrow["{\psi_{2,1}}"', from=2-6, to=2-5]
		\arrow["{\psi_{1,1}}"', from=2-7, to=2-6]
		\arrow[dashed, no head, from=3-4, to=1-4]
		\arrow["{\psi_{2,1}}", from=3-4, to=2-3]
		\arrow["{\phi_{2,1}}"', from=3-4, to=2-5]
	\end{tikzcd}\]
\end{minipage}
\hfill
\begin{minipage}{0.3\textwidth}
	$\alpha \defeq \left( \begin{smallmatrix} 1 & 1 & 0 \\ 0 & 1 & 1 \end{smallmatrix} \right)$
	\smallskip
	
	$\beta_x \defeq \left( \begin{smallmatrix} x & 1 & 1 \\ 1 & 0 & 1	\end{smallmatrix} \right)$
\end{minipage}
\smallskip

\begin{minipage}{0.45\textwidth}
	\noindent $(\qR_4)$
	\[\begin{tikzcd}[ampersand replacement=\&,sep=tiny]
		\&\& {\K(x)^3} \&\&\&\&\& \\
		{\K(x)^2} \& {\K(x)^4} \&\& {\K(x)^5} \& {\K(x)^4} \& {\K(x)^3} \& {\K(x)^2} \& {\K(x)} \\
		\&\& {\K(x)^3}
		\arrow["{\psi_ {2,2}}", from=2-1, to=2-2]
		\arrow["{\beta_x}", color={rgb,255:red,255;green,51;blue,54}, from=2-2, to=1-3]
		\arrow["\alpha"', from=2-4, to=1-3]
		\arrow["{\phi_ {4,1}}"', from=2-5, to=2-4]
		\arrow["{\phi_ {3,1}}"', from=2-6, to=2-5]
		\arrow["{\phi_ {2,1}}"', from=2-7, to=2-6]
		\arrow["{\phi_ {1,1}}"', from=2-8, to=2-7]
		\arrow[dashed, no head, from=3-3, to=1-3]
		\arrow["{\psi_ {3,1}}", from=3-3, to=2-2]
		\arrow["{\psi_ {3,2}}"', from=3-3, to=2-4]
	\end{tikzcd}\]
\end{minipage}
\hfill
\begin{minipage}{0.2\textwidth}
	$\alpha \defeq \left( \begin{smallmatrix}
		\begin{smallmatrix} 0 & 1 \\ 1 & 0 \\ 1 & 1 \end{smallmatrix} & \boldsymbol{1}_3 \end{smallmatrix} \right)$
	\smallskip
	
	$\beta_x \defeq \left( \begin{smallmatrix} \boldsymbol{1}_3 & \begin{smallmatrix} 0 & x \\ 0 & 1 \\ 1 & 0	\end{smallmatrix} \end{smallmatrix} \right)$
\end{minipage}
\smallskip

\begin{minipage}{0.45\textwidth}
	\noindent $(\qR_5)$
	\[\begin{tikzcd}[ampersand replacement=\&,sep=tiny]
		\&\&\&\&\& {\K(x)^3} \& \\
		{\K(x)} \& {\K(x)^2} \& {\K(x)^3} \& {\K(x)^4} \& {\K(x)^5} \&\& {\K(x)^2} \\
		\&\&\&\& {\K(x)^3} \\
		\&\&\&\&\& {\K(x)}
		\arrow["{\phi_{1,1}}", from=2-1, to=2-2]
		\arrow["{\phi_{2,1}}", from=2-2, to=2-3]
		\arrow["{\phi_{3,1}}", from=2-3, to=2-4]
		\arrow["{\phi_{4,1}}", from=2-4, to=2-5]
		\arrow["{\beta'_x}", color={rgb,255:red,255;green,51;blue,54}, from=2-5, to=1-6]
		\arrow["{\beta_x}"', color={rgb,255:red,255;green,51;blue,54}, from=2-7, to=1-6]
		\arrow["{\psi_{3,2}}", from=3-5, to=2-5]
		\arrow[dashed, no head, from=4-6, to=1-6]
		\arrow["{\psi_{1,1}}"', from=4-6, to=2-7]
		\arrow["{\psi_{1,2}}", from=4-6, to=3-5]
	\end{tikzcd}\]
\end{minipage}
\hfill
\begin{minipage}{0.3\textwidth}
	$\beta_x \defeq \left( \begin{smallmatrix} 1 & x \\ 1 & 1 \\ 0 & 1 \end{smallmatrix} \right)$
	\smallskip
	
	$\beta'_x \defeq \left( \begin{smallmatrix} \boldsymbol{1}_3 & \begin{smallmatrix} 0 & x \\ 1 & 1 \\ 0 & 1	\end{smallmatrix} \end{smallmatrix} \right)$
\end{minipage}
\smallskip

\begin{minipage}{0.45\textwidth}
	\noindent $(\qR_6)$
	\[\begin{tikzcd}[ampersand replacement=\&,sep=tiny]
		\&\& {\K(x)^2} \&\&\&\& \\
		\& {\K(x)^3} \\
		{\K(x)^2} \& {\K(x)^4} \&\& {\K(x)^4} \& {\K(x)^3} \& {\K(x)^2} \& {\K(x)} \\
		\&\& {\K(x)^3}
		\arrow["{\alpha_2}", from=2-2, to=1-3]
		\arrow["{\psi_{2,2}}", from=3-1, to=3-2]
		\arrow["{\alpha_1}", from=3-2, to=2-2]
		\arrow["{\beta_x}"', color={rgb,255:red,255;green,51;blue,51}, from=3-4, to=1-3]
		\arrow["{\phi_{3,1}}"', from=3-5, to=3-4]
		\arrow["{\phi_{2,1}}"', from=3-6, to=3-5]
		\arrow["{\phi_{1,1}}"', from=3-7, to=3-6]
		\arrow[dashed, no head, from=4-3, to=1-3]
		\arrow["{\phi_{3,1}}", from=4-3, to=3-2]
		\arrow["{\psi_{3,1}}"', from=4-3, to=3-4]
	\end{tikzcd}\]
\end{minipage}
\hfill
\begin{minipage}{0.3\textwidth}
	$\alpha_1 \defeq \left( \begin{smallmatrix} \boldsymbol{1}_3 & \begin{smallmatrix}
			0 \\ 1 \\ 0
	\end{smallmatrix} \end{smallmatrix} \right)$ 
	\smallskip
	
	$\alpha_2 \defeq \left( \begin{smallmatrix} \boldsymbol{1}_2 & \begin{smallmatrix}
			1 \\ 1
	\end{smallmatrix} \end{smallmatrix} \right)$
	\smallskip
	
	$\beta_x \defeq \left( \begin{smallmatrix}  x & 1 & 0 & 1 \\ 1 & 0 & 1 & 1	\end{smallmatrix} \right)$
\end{minipage}
\smallskip

\begin{minipage}{0.45\textwidth}
	\noindent $(\qR_7)$
	\[\begin{tikzcd}[ampersand replacement=\&,sep=tiny]
		\&\& {\K(x)^2} \&\&\& \\
		\& {\K(x)^3} \\
		{\K(x)^2} \& {\K(x)^4} \&\& {\K(x)^3} \& {\K(x)^2} \& {\K(x)} \\
		\& {\K(x)^3} \\
		\&\& {\K(x)^2}
		\arrow["{\alpha_2}", from=2-2, to=1-3]
		\arrow["{\psi_{2,2}}", from=3-1, to=3-2]
		\arrow["{\alpha_1}", from=3-2, to=2-2]
		\arrow["{\beta_x}"', color={rgb,255:red,255;green,51;blue,51}, from=3-4, to=1-3]
		\arrow["{\phi_{2,1}}"', from=3-5, to=3-4]
		\arrow["{\phi_{1,1}}"', from=3-6, to=3-5]
		\arrow["{\phi_{3,1}}", from=4-2, to=3-2]
		\arrow[dashed, no head, from=5-3, to=1-3]
		\arrow["{\psi_{2,1}}"', from=5-3, to=3-4]
		\arrow["{\phi_{2,1}}", from=5-3, to=4-2]
	\end{tikzcd}\]
\end{minipage}
\hfill
\begin{minipage}{0.3\textwidth}
	$\alpha_1 \defeq \left( \begin{smallmatrix} \boldsymbol{1}_3 & \begin{smallmatrix}
			0 \\ 1 \\ 0
	\end{smallmatrix} \end{smallmatrix} \right)$ 
	\smallskip
	
	$\alpha_2 \defeq \left( \begin{smallmatrix} \boldsymbol{1}_2 & \begin{smallmatrix}
			1 \\ 1
	\end{smallmatrix} \end{smallmatrix} \right)$
	\smallskip
	
	$\beta_x \defeq \left( \begin{smallmatrix}  x & 1 & 0 \\ 1 & 0 & 1	\end{smallmatrix} \right)$
\end{minipage}

	\printbibliography

\end{document}